\documentclass[12pt,a4paper,twoside,reqno]{amsart}
\title[Futaki Invariants, Yau's Conjecture, Hull-Strominger]{Futaki Invariants and Yau's Conjecture on the Hull-Strominger system}
\author[M. Garcia-Fernandez]{Mario Garcia-Fernandez}
\address{Universidad Aut\'onoma de Madrid and Instituto de Ciencias Matem\'aticas (CSIC-UAM-UC3M-UCM)\\ Ciudad
	Universitaria de Cantoblanco\\ 28049 Madrid, Spain}
\email{mario.garcia@icmat.es}

\author[R. Gonzalez Molina]{Raul Gonzalez Molina}
\address{Instituto de Ciencias Matem\'aticas (CSIC-UAM-UC3M-UCM)\\ Nicol\'as Cabrera 13--15, Cantoblanco\\ 28049 Madrid, Spain}
  \email{raul.gonzalez@icmat.es}

\thanks{This work is partially supported by the Spanish Ministry of Science and Innovation, under grants PID2019-109339GA-C32, EUR2020-112265, and I+D+i SEV2015-05554-18-1 funded by MCIN/AEI/10.13039/501100011033}

\usepackage{hyperref,url}
\usepackage{amssymb,latexsym}
\usepackage{amsmath,amsthm}
\usepackage[latin1]{inputenc}
\usepackage{enumerate}
\usepackage{graphicx}
\usepackage[dvipsnames]{xcolor}
\usepackage[all]{xy} \CompileMatrices
\SelectTips{cm}{12} % computer modern fonts of size 12 for arrowheads
\theoremstyle{plain}
\newtheorem{theorem}{Theorem}[section]
\newtheorem{lemma}[theorem]{Lemma}

\newtheorem{proposition}[theorem]{Proposition}
\newtheorem{conjecture}[theorem]{Conjecture}
\newtheorem{question}[theorem]{Question}
\theoremstyle{definition}
\newtheorem{definition}[theorem]{Definition}
\newtheorem{definition-theorem}[theorem]{Definition-Theorem}
\newtheorem{example}[theorem]{Example}
\theoremstyle{remark}
\newtheorem{remark}[theorem]{Remark}

\numberwithin{equation}{section} 

\newcommand{\tr}{\operatorname{tr}}
\newcommand{\End}{\operatorname{End}}
\newcommand{\Hom}{\operatorname{Hom}}
\newcommand{\Ker}{\operatorname{Ker}}
\newcommand{\ad}{\operatorname{ad}}

\newcommand{\dbar}{\bar{\partial}}

\newcommand{\CC}{{\mathbb C}}

\newcommand{\RR}{{\mathbb R}}

\newcommand{\surj}{\to\kern-1.8ex\to}

\newcommand{\IP}[1]{\left<#1\right>}

\begin{document}

\maketitle

\begin{abstract}%

We find a new obstruction to the existence of solutions of the Hull-Strominger system, which goes beyond the balanced property of the Calabi-Yau manifold $(X,\Omega)$ and the Mumford-Takemoto slope stability of the bundle over it. The basic principle is the construction of a (possibly indefinite) Hermitian-Einstein metric on the holomorphic string algebroid associated to a solution of the system, provided that the connection $\nabla$ on the tangent bundle is Hermitian-Yang-Mills. Using this, we define a family of Futaki invariants obstructing the existence of solutions in a given balanced class. %, whereby we find balanced classes and polystable bundles on non-K\"ahler Calabi-Yau threefolds, with holomorphic string algebroids which do not admit solutions of the Hull-Strominger system. 
Our results are motivated by a strong version of a conjecture by Yau on the existence problem for these equations.

\end{abstract}

%\tableofcontents

\section{Introduction}\label{sec:intro}

Let $X$ be a compact complex manifold of dimension three endowed with a holomorphic volume form $\Omega$. Let $V$ be a holomorphic vector bundle over $X$. Let $\alpha$ be a real constant. The Hull-Strominger system, for Hermitian metrics $g$ on $X$ and $h$ on $V$, is given by
\begin{equation}\label{eq:HSintro}
\begin{split}
F_h \wedge \omega^2 & = 0,\\
d(\|\Omega\|_\omega \omega^2) & = 0,\\
dd^c \omega - \alpha(\tr R_\nabla \wedge R_\nabla - \tr F_h \wedge F_h) & = 0,
\end{split}
\end{equation} 
where $\omega$ denotes the Hermitian form of $g$. The first line in \eqref{eq:HSintro} is the Hermitian-Einstein equation for the metric $h$. The conformally balanced condition, in the second line, is equivalent to the holonomy of the Bismut connection of $g$ being contained in $\operatorname{SU}(3)$. In the last equation, known as the Bianchi identity, there is an ambiguity in the choice of a metric connection $\nabla$ in the tangent bundle of the manifold, back to its origins in heterotic string theory \cite{HullTurin,Strom}. %Mathematically, \eqref{eq:HS} can be regarded as a family of systems of partial differential equations, often called \emph{Hull-Strominger systems}, parametrized by the different \emph{ans\"atze} for the connection $\nabla$. Unless otherwise stated, in this paper we will consider only the case $n=3$.

The Hull-Strominger system has recently generated a great deal of interest in mathematics, both for its applications to the study of non-K\"ahler Calabi-Yau manifolds \cite{FeiReview,GFReview,PPZ3} and its relation to a conjectural generalization of mirror symmetry \cite{AAGa,Yau2005}. As originally proposed in the seminal work by Li-Yau \cite{LiYau} and Fu-Yau \cite{FuYau2,FuYau} on these equations, it is expected that the Hull-Strominger system plays a key role on the geometrization of \emph{Reid's fantasy} \cite{CoPiYau,FuLiYau}, connecting complex threefolds with trivial canonical bundle via conifold transitions. This proposal has important implications in our understanding of the moduli space of projective Calabi-Yau manifolds in complex dimension three, and also physical applications to the \emph{string landscape}. 

The existence and uniqueness problem for the Hull-Strominger is currently widely open. The present work is motivated by a conjecture about the existence of solutions by S.-T. Yau \cite{Yau2006}, which provides an important test for the general methods recently developed in \cite{GaRuShTi,Phong}. In order to understand %\emph{Yau's Conjecture on the Hull-Strominger system}
this question, let us first explain three types of necessary conditions for \eqref{eq:HSintro}. Firstly, for $(X,\Omega,V)$ to admit a solution there are some evident cohomological obstructions on the Chern classes, namely,
\begin{equation}\label{eq:c1c2}
\deg_{\mathfrak{b}}(V) = 0, \qquad  ch_2(V) = ch_2(X) \in H^{2,2}_{BC}(X,\RR),
\end{equation}
where $ch_2$ denotes the second Chern character, $\deg_{\mathfrak{b}}(V):= c_1(V) \cdot \mathfrak{b} \in \mathbb{R}$, and
\begin{equation}\label{eq:balancedclass}
\mathfrak{b} := [\|\Omega\|_\omega \omega^2] \in H^{2,2}_{BC}(X,\RR).
\end{equation}
Here $H^{p,q}_{BC}(X)$ are the Bott-Chern cohomology groups of the complex manifold $X$, 
%defined by
%\begin{equation}\label{eq:BC}
%H^{p,q}_{BC}(X)=\frac{\mathrm{ker} \hspace{1mm} d: \Omega^{p,q}(X,\mathbb{C}) \longrightarrow \Omega^{p+q+1}(X,\mathbb{C})}{\mathrm{Im} \hspace{1mm}  dd^{c}: \Omega^{p-1,q-1}(X,\mathbb{C}) \longrightarrow \Omega^{p,q}(X,\mathbb{C})}
%\end{equation}
and $H^{p,p}_{BC}(X,\mathbb{R}) \subset H^{p,p}_{BC}(X)$ is the canonical real structure. Secondly, the Calabi-Yau threefold $(X,\Omega)$ must be balanced  \cite{Michel}, that is, it has to admit a balanced Hermitian metric with balanced class $\mathfrak{b}$. 
%Recall that this property can be expressed as a positivity condition on the homology of $X$ %which involves complex currents 
%(see \cite{Michel}).
%Recently, the balanced property of a metric has been interpreted as a moment map condition using generalized geometry in \cite{GaRuTi3}.
Thirdly, the Donaldson-Uhlenbeck-Yau Theorem \cite{Don,UYau} and its extensions to Hermitian manifolds (see \cite{Buchdahl,LiYauHYM,lt}) requires the Mumford-Takemoto slope stability of $V$ with respect to $\mathfrak{b}$.

In this setup, Yau's Conjecture for the Hull-Strominger system \cite{Yau2006} states that the previous three necessary conditions are the only obstruction to find a solution of the Hull-Strominger system (see Conjecture \ref{conj:Yau}). In order to make progress in this interesting question, about which we know very little at present, it is natural to strenghten the statement in two ways. On the one hand,  it is desirable that a complete answer to the existence problem has control on the balanced class $\mathfrak{b}$, producing a solution of the Hull-Strominger with the same balanced class that we use, a priori, to measure the stability of the bundle $V$ (without this assumption, an affirmative answer to Conjecture \ref{conj:Yau} was given in \cite{AGF1}, for $V$ stable with respect to a K\"ahler class). The existence of solutions when $\mathfrak{b}$ is prescribed to be the square of a K\"ahler class has been recently established by Collins, Picard, and Yau in \cite{CoPiYau2}. Note, however, that, even assuming that $X$ is projective, these classes do not exhaust the balanced cone due to a result of Fu and Xiao \cite{FuXiao}. %, who proved that there exists nef and big classes on a projective Calabi-Yau threefold which square to a balanced class. 

On the other hand, as originally formulated in \cite{Yau2006}, the connection $\nabla$ in \eqref{eq:HSintro} is not specified in the statement of Conjecture \ref{conj:Yau}. In a later formulation \cite{Yau2}, $\nabla$ is taken to be the Chern connection of $g$ following a suggestion in \cite{Strom}. However, a different choice of connection seems to have both more physical and geometrical significance, namely, to take $\nabla$ satisfying the Hermitian-Yang-Mills equations
\begin{equation}\label{eq:HYMintro}
R_\nabla^{0,2} = 0, \qquad R_\nabla \wedge \omega^{2} = 0.
\end{equation}
With this ansatz for $\nabla$, a solution of \eqref{eq:HSintro} solves the \emph{heterotic equations of motion} \cite{FeIvUgVi,Ivan09} and furthermore has many desirable properties in perturbation theory \cite{OssaSvanes,Hull2,MaSp}. As for the geometry, solutions of \eqref{eq:HSintro} satisfying \eqref{eq:HYMintro} are \emph{generalized Ricci flat} \cite{Ga0,Ga1} and have a moment map interpretation \cite{Waldram,GaRuTi3}, which leads to an interesting metric on its moduli space. Furthermore, there is currently strong evidence that these solutions play an important role in \emph{(0,2) mirror symmetry} via T-duality and the theory of vertex algebras \cite{AAGa,Arriba,Ga2}.

Motivated by the previous discussion, in the present work we propose to address the following strong version of Yau's Conjecture \ref{conj:Yau} with the ansatz \eqref{eq:HYMintro}. Notice that a solution of \eqref{eq:HYMintro} induces a (possibly) non-standard structure of holomorphic vector bundle on $T^{1,0} = T^{1,0}X$, which we denote $V_0 = (T^{1,0},\nabla^{0,1})$. Let $h^{0,1}_{\dbar}(X)$ be the dimension of the Dolbeault cohomology group $H^{0,1}_{\dbar}(X)$.

\begin{question}\label{question}
Let $(X,\Omega)$ be a compact Calabi-Yau threefold with $h^{0,1}_{\dbar}(X) = 0$ and endowed with a balanced class $\mathfrak{b}$. Let $V$ be a $\mathfrak{b}$-polystable holomorphic vector bundle over $X$ satisfying \eqref{eq:c1c2}. Let $V_0$ be a generic $\mathfrak{b}$-polystable holomorphic vector bundle structure on $T^{1,0}$. Does $(X,\Omega,V)$ admit a solution $(g,h)$ of the Hull-Strominger system \eqref{eq:HSintro} with $\alpha \neq 0$ and balanced class $\mathfrak{b}$, such that $\nabla$ is the Chern connection of a Hermitian-Einstein metric $h_0$ on $V_0$? 
%satisfying \eqref{eq:HYMintro} and furthermore $V_0 = (T^{1,0}X,\nabla^{0,1})$?
\end{question}

For non-generic $V_0$ or $h^{0,1}_{\dbar}(X) \neq 0$, it is not difficult to find examples which establish a negative solution to the previous question (see Section \ref{sec:Yaustrong}). The idea of these technical conditions is to prevent the existence of `degenerate' solutions of the Bianchi identity, that is, with $\tr R_\nabla \wedge R_\nabla = \tr F_h \wedge F_h$, which cannot exist when $X$ is a non-K\"ahler manifold. We speculate that these obstructions arise, indeed, from `categorical symmetries' of the combined equations \eqref{eq:HSintro} and \eqref{eq:HYMintro}, %(see Remark \ref{rem:category}) 
which should find an explanation via higher gauge theory \cite{AGaT}.

%Observe that an affirmative answer to Question \ref{question} provides, in particular, a solution of Conjecture \ref{conj:Yau} with the ansatz \eqref{eq:HYMintro} (see Remark \ref{rem:HSabstract}). It is an open question whether, assuming that the holomorphic tangent bundle $T^{1,0}$ of $X$ is $\mathfrak{b}$-polystable, one can reduce Yau's Conjecture \ref{conj:Yau} for $\nabla$ the Chern connection of $g$, as originally proposed in \cite{FuYau2,FuYau,Yau2}, to Question \ref{question}. 

\subsection{Futaki invariants} 

The main goal of the present work is to provide a new tool which help us to address the strong version of Yau's Conjecture formulated in Question \ref{question}. In order to do this, we will exploit the special features of the solutions of the Hull-Strominger system with the ansatz \eqref{eq:HYMintro}. More precisely, 
%The Hermitian-Yang-Mills condition for the connection $\nabla$ will enable us 
we will be able to use generalized geometry and to apply the theory of metrics on holomorphic string algebroids introduced in \cite{GaRuShTi,GaRuTi2}.

Let $(X,\Omega)$ be a compact Calabi-Yau manifold endowed with a pair of holomorphic vector bundles $V_0$ and $V$, as in Question \ref{question}. For simplicity, in this introduction we will assume that $\mathrm{dim}_\mathbb{C} X=3$, but our results work also for an abstract definition of the Hull-Strominger system in higher dimensions (see Definition \ref{def:HSabstract}). Denote by $P$ the holomorphic principal bundle of split frames of $V_0 \oplus V$. Then, $P$ satisfies
\begin{equation}\label{eq:BC22intro}
p_1(P) = 0 \in H^{2,2}_{BC}(X,\RR),
\end{equation}
%Then, there exists a solution of
%$$
%dd^c \tau_0 - \IP{F_{h_0} \wedge F_{h_0}} = 0
%$$
%for some $\tau_0 \in \Omega^{1,1}_\RR$ and some reduction $h_0$ to a maximal compact and 
where $p_1(P)$ denotes the \emph{first Pontryagin class} associated via Chern-Weyl Theory to the bilinear form
\begin{equation}\label{eq:pairingHSintro}
\IP{,} : = - \alpha \tr_{V_0} + \alpha \tr_{V}.
\end{equation}
Using \eqref{eq:BC22intro}, one can canonically associate to $P$ a family of holomorphic vector bundle extensions of the form
\begin{equation}\label{eq:holCoustrintro}
0 \longrightarrow T^*_{1,0} \longrightarrow \mathcal{Q} \longrightarrow A_P \longrightarrow 0,
\end{equation}
where $A_P$ denotes the holomorphic Atiyah algebroid of $P$. In the language of \cite{GaRuShTi,GaRuTi2}, these extensions correspond to the underlying orthogonal bundle of a particular class of holomorphic Courant algebroids, called \emph{string}. %In \cite[Proposition 3.6]{GaRuShTi} it was proved that any solution of the Hull-Strominger system satisfying \eqref{eq:HYMintro} defines a holomorphic string algebroid of special type. 
Consider the Aeppli cohomology groups of the complex manifold $H^{p,q}_{A}(X)$, 
%\begin{equation}\label{eq:A}
%H^{p,q}_{A}(X)=\frac{\mathrm{ker} \hspace{1mm} \partial\overline{\partial}: \Omega^{p,q}(X,\mathbb{C}) \longrightarrow \Omega^{p+1,q+1}(X,\mathbb{C})}{\mathrm{Im} \hspace{1mm}  \partial \oplus \overline{\partial}: \Omega^{p-1,q}(X,\mathbb{C}) \oplus \Omega^{p,q-1}%(X,\mathbb{C}) \longrightarrow \Omega^{p,q}(X,\mathbb{C})}
%\end{equation}
and the canonical real structures $H^{p,p}_A(X,\RR) \subset H^{p,p}_A(X)$. Consider the map to the Dolbeault cohomology of $X$ induced by the $\partial$ operator:
$$
\partial \colon H^{1,1}_A(X,\RR) \to H^{2,1}_{\dbar}(X).
$$
Then, the holomorphic vector bundle extensions \eqref{eq:holCoustrintro} arising from the string algebroids of our interest are parametrized by an affine space $\mathfrak{S}$ modelled over the vector space $\operatorname{Im} \partial$ (see Proposition \ref{prop:BCclassification} and Remark \ref{rem:forgetful}).  

Consider the family of finite-dimensional complex Lie algebras
$$
\mathfrak{H} \to \mathfrak{S}
$$
where the fibre over $\mathfrak{s} \in \mathfrak{S}$ is given by the Lie algebra of the group of holomorphic gauge transformations of the $SO(r,\mathbb{C})$-bundle $\mathcal{Q}_{\mathfrak{s}}$
$$
\mathfrak{H}_\mathfrak{s} := H^0(X,\End \mathcal{Q}_\mathfrak{s}).
$$
Then, there is a family of characters (see Section \ref{sec:Futakidef})
\begin{equation*}%\label{eq:FcharacterAeintro}
    \xymatrix{
 \mathcal{F} \colon \mathfrak{H} \ar[r] & H^{1,1}_A(X).
% \\
 %   H^{1,1}_A(X,\RR) &
 }
\end{equation*}
More explicitly, for a real balanced class $\mathfrak{b} \in H^{2,2}_{BC}(X,\RR)$ and $\mathfrak{s} \in \mathfrak{S}$, the Lie algebra character
$$
\IP{\mathcal{F}_{\mathfrak{s}},\mathfrak{b}} \colon \mathfrak{H}_{\mathfrak{s}} \to \mathbb{C}.
$$
is given by a Futaki invariant for the orthogonal bundle $\mathcal{Q}_{\mathfrak{s}}$ defined by $\mathfrak{b}$, where we use the duality isomorphism $H^{2,2}_{BC}(X)^* \cong H^{1,1}_A(X)$.

\begin{theorem}\label{t:Futakiintro}
Let $(X,\Omega)$ be a Calabi-Yau threefold endowed with a pair of holomorphic vector bundles $V_0$ and $V$ satisfying \eqref{eq:c1c2}. Assume that $(X,\Omega,V)$ admits a solution of the Hull-Strominger system \eqref{eq:HSintro} with balanced class $\mathfrak{b} \in H^{2,2}_{BC}(X,\RR)$, such that $\nabla$ is the Chern connection of a Hermitian-Einstein metric $h_0$ on $V_0$. Then, there exists $\mathfrak{s} \in\mathfrak{S}$ such that $\IP{\mathcal{F}_\mathfrak{s},\mathfrak{b}} = 0$.
\end{theorem}

This gives a new potential obstruction to the existence of solutions of the Hull-Strominger system with the ansatz \eqref{eq:HYMintro}, which goes beyond the balanced property of the Calabi-Yau manifold $(X,\Omega)$ and the Mumford-Takemoto slope stability of the bundles $V_0$ and $V$.

As a consequence of our main result, in order to find a negative answer to Question \ref{question} it suffices to find a tuple $(X,\Omega,V)$ and a balanced class $\mathfrak{b} \in H^{2,2}_{BC}(X,\RR)$ as in the statement, such that $V$ is $\mathfrak{b}$-polystable and 
\begin{equation*}%\label{eq:obstructionintro}
\IP{\mathcal{F}_\mathfrak{s},\mathfrak{b}} \neq 0, \qquad \forall \mathfrak{s} \in \mathfrak{U}^0,
\end{equation*}
where $\mathfrak{U}^0$ denotes the restriction of the relative family of string algebroid extensions over a dense open subset of the moduli space for $V_0$. Despite our efforts, we have not been able to find an example which fulfills this strong condition, and hence Question \ref{question} is still open. Still, we hope that our new invariant should be able to decide the non-existence of solutions for a specific choice of $V_0$, a question we plan to address in future work.
%We think that it would be worth investigating this question via the geometric flow methods introduced in \cite{Phong}. In particular, relating \eqref{eq:obstruction} with the formation of singularities of the specific version of the anomaly flow in \cite[Section 4.4]{PPZ3}.

When the Calabi-Yau manifold $X$ satisfies the $\partial\dbar$-Lemma the space $\mathfrak{S}$ reduces to a point, that is, $\mathcal{Q}$ in \eqref{eq:holCoustrintro} is uniquely determined by $P$ up to isomorphism \cite{GaRuShTi}. In this case we obtain a unique invariant $\mathcal{F}_0$ obstructing the existence of solutions, which can be regarded as a \emph{stringy version} of the classical Futaki invariant for the holomorphic bundle $P$. Based on this, we expect that $\mathcal{F}_0$ provides an efficient tool to address Question \ref{question} in the case of Calabi-Yau manifolds obtained via conifold transitions and flops, with potential interesting implications in the geometrization of Reid's fantasy and the string landscape. It is interesting to observe that our invariant $\mathcal{F}_0$ identically vanishes on a general Clemens-Friedman manifold, since the balanced cone reduces to the zero class \cite{Friedman2019}. This is in agreement with the expectation that these non-K\"ahler manifolds carry solutions of the Hull-Strominger system. 

\subsection{Coupled Hermitian-Einstein metrics}

Our method of proof of Theorem \ref{t:Futakiintro} has several interesting salient features. It is inspired by an important result by De La Ossa, Larfors, and Svanes \cite{OssaLarforsSvanes}, who showed that the Hull-Strominger system is equivalent to a suitable Hermitian-Yang-Mills equation on a Courant algebroid to all orders in perturbation theory. Section \ref{sec:HE} is devoted to give a precise mathematical counterpart of their result (see Remark \ref{rem:DelaOssa}): in Section \ref{sec:BHE} we characterize the Hermitian-Einstein condition
\begin{equation}\label{eq:HEintro}
F_\mathbf{G} \wedge \omega^{n-1} = 0,
\end{equation}
for a generalized pseudo-Hermitian metric $\mathbf{G}$ on a holomorphic string algebroid (see Definition \ref{d:generalizedmetriHerm}) in terms of classical tensors. Using this, in Proposition \ref{prop:HS-HE} we prove that any solution of the Hull-Strominger system with the ansatz \eqref{eq:HYMintro} induces a solution of \eqref{eq:HEintro}, which allows us to construct the Futaki invariants in Theorem \ref{t:Futakiintro}.

Equation \eqref{eq:HEintro} motivates the definition of a new system of coupled equations, which we call the \emph{coupled Hermitian-Einstein system} (see Definition \ref{def:BHEsystem}), inducing solutions of \eqref{eq:HEintro}. The coupled Hermitian-Einstein system is weaker and more flexible than the Hull-Strominger system as, for instance, does not require the complex manifold to be balanced neither to have a holomorphic volume form. Hence, we expect to find solutions on compact complex manifolds with $c_1(X) = 0$ but non-trivial canonical bundle. Furthermore, its solutions correspond to a natural class of generalized Ricci flat metrics on string algebroids and exhibit an interesting relation to heterotic supergravity, giving further motivation for their study (see Remark \ref{rem:GRF} and Remark \ref{rem:heterotic}). This problem will be investigated elsewhere.

The aforementioned ideas and methods of the present paper are exploded further in the sequel \cite{GFGM}, where we investigate stability conditions for the holomorphic string algebroid associated to a solution of the Hull-Strominger system with the ansatz \eqref{eq:HYMintro}. Even though our picture is mostly conjectural, we expect that this stability condition will lead us to new obstructions to the existence of solutions in future studies, similarly as in \cite{GaJoSt}.

\section{Strengthening Yau's Conjecture}\label{sec:YauConj}

\subsection{Original statement}

In this section we present a precise statement of Yau's Conjecture for the Hull-Strominger system \ref{eq:HSintro}, as originally stated in \cite{Yau2006}, and comment on the strong version that we propose in Question \ref{question}, with the ansatz \ref{eq:HYMintro}. We start by introducing some notation. Given a compact complex manifold $X$, we will denote by $H^{p,q}_{BC}(X)$ its Bott-Chern cohomology groups, defined by
\begin{equation}\label{eq:BC}
H^{p,q}_{BC}(X)=\frac{\mathrm{ker} \hspace{1mm} d: \Omega^{p,q}(X,\mathbb{C}) \longrightarrow \Omega^{p+q+1}(X,\mathbb{C})}{\mathrm{Im} \hspace{1mm}  dd^{c}: \Omega^{p-1,q-1}(X,\mathbb{C}) \longrightarrow \Omega^{p,q}(X,\mathbb{C})},
\end{equation}
and by $H^{p,p}_{BC}(X,\mathbb{R}) \subset H^{p,p}_{BC}(X)$ the canonical real structure. We will say that $X$ is balanced if it admits a balanced metric, that is, a Hermitian metric $g_0$ with Hermitiam form $\omega_0 = g_0(J,)$, such that $d \omega_0^{n-1} = 0$. In this case, $\omega_0$ defines a balanced class \begin{equation}\label{eq:balancedclass0}
\mathfrak{b}_0 := [\omega^{n-1}_0] \in H^{n-1,n-1}_{BC}(X,\RR).
\end{equation}
Furthermore, we will say that $X$ is a Calabi-Yau manifold if it admits a holomorphic volume form $\Omega$, which will be fixed in the sequel. Observe that we do not require that $X$ is projective, neither K\"ahler.

Let $V$ a holomorphic vector bundle over a balanced manifold $X$, with balanced class $\mathfrak{b}_0 = [\omega^{n-1}_0]$. Then, by the Donaldson-Uhlenbeck-Yau Theorem \cite{Don,UYau} and its extensions to Hermitian manifolds (see \cite{Buchdahl,LiYauHYM,lt}) the existence of solutions of the Hermitian-Einstein equation
$$
F_h \wedge \omega_0^{n-1} = 0,
$$
for a Hermitian metric $h$ on $V$, is equivalent to the Mumford-Takemoto slope polystability of $V$ with respect to $\mathfrak{b}_0$.

With these preliminaries, Yau's Conjecture for the Hull-Strominger system can be stated as follows:

\begin{conjecture}[Yau's Conjecture \cite{Yau2006}]\label{conj:Yau}
Let $(X,\Omega)$ be a compact Calabi-Yau threefold endowed with a balanced class $\mathfrak{b}_0$. Let $V$ be a holomorphic vector bundle over $X$ satisfying \eqref{eq:c1c2}. If $V$ is polystable with respect to $\mathfrak{b}_0$, then $(X,\Omega,V)$ admits a solution of the Hull-Strominger system \eqref{eq:HSintro}.
\end{conjecture}

As mentioned in Section \ref{sec:intro}, it is natural to strenghten the statement of Yau's Conjecture in two ways. Firstly, Conjecture \ref{conj:Yau} does not specify whether $\mathfrak{b}_0$ equals the balanced class of the solution $\mathfrak{b}$ defined in \eqref{eq:balancedclass}. Hence, it is desirable that a complete answer to Conjecture \ref{conj:Yau} has control on the balanced class, producing a solution of the Hull-Strominger with $\mathfrak{b}_0 = \mathfrak{b}$. Secondly, as originally formulated in \cite{Yau2006}, the connection $\nabla$ in \eqref{eq:HSintro} is not specified in the statement of Conjecture \ref{conj:Yau}. Here, we propose to take $\nabla$ satisfying the Hermitian-Yang-Mills equations \eqref{eq:HYMintro}, both for its physical and geometrical significance (see Section \ref{sec:intro}).

Our approach to the existence problem for the Hull-Strominger system, with $\nabla$ satisfying \eqref{eq:HYMintro}, lead us to consider holomomorphic vector bundle structures $V_0$ on $T^{1,0}$ which are polystable with respect to the balanced class $\mathfrak{b}$. For special choices of $V$ and $V_0$, however, one may find solutions of the Hermitian-Einstein equation such that $\tr R_\nabla \wedge R_\nabla = \tr F_h \wedge F_h$, which will obstruct the existence of solutions to the Bianchi identity when $X$ is a non-K\"ahler manifold. As we will see more explicitly in the next section, this motivates the statement of Question \ref{question}. 

\begin{remark}
Observe that an affirmative answer to Question \ref{question} provides, in particular, a solution to Conjecture \ref{conj:Yau} with the ansatz \eqref{eq:HYMintro} (see Remark \ref{rem:HSabstract}). It is an open question whether, assuming that the holomorphic tangent bundle $T^{1,0}$ of $X$ is $\mathfrak{b}$-polystable, one can reduce Yau's Conjecture \ref{conj:Yau} for $\nabla$ the Chern connection of $g$, as proposed in \cite{FuYau2,FuYau,Yau2}, to Question \ref{question}.
\end{remark}

\subsection{The genericity and $h^{0,1}_{\dbar}(X) \neq 0$ conditions}\label{sec:Yaustrong}

The aim of this section is to explain the necessity of the technical conditions in the statement of Question \ref{question}, namely, the genericity of the holomomorphic vector bundle structure $V_0$ and $h^{0,1}_{\dbar}(X) \neq 0$. These are motivated by the following result.

\begin{proposition}\label{prop:V0h01}
Let $(X,\Omega)$ be a compact Calabi-Yau threefold endowed with a balanced class $\mathfrak{b}$. Assume that $X$ does not admit any K\"ahler metric. Let $L \to X$ be a holomorphic line bundle on $X$ with vanishing first Chern class $c_1(L) = 0$. Let $V_0$ be a holomorphic bundle structure on $T^{1,0}$ which is polystable with respect to $\mathfrak{b}$. Then, $(X,\Omega,V_0\otimes L)$ does not admit a solution of the Hull-Strominger system \eqref{eq:HSintro} with balanced class $\mathfrak{b}$, such that $\nabla$ is the Chern connection of a Hermitian-Einstein metric $h_0$ on $V_0$.
\end{proposition}

\begin{proof}
Assume that $(g,h,\nabla)$ is such a solution, for $\nabla$ the Chern connection of $h_0$. Let $h_L$ be a flat metric on $L$. Then, $h_0 \otimes h_L$ is a Hermitian-Einstein metric on $V_0\otimes L$, and therefore there exists a holomorphic gauge transformation taking $h$ to $h_0 \otimes h_L$. In particular, since $h_L$ is flat, one has
$$
\tr F_h^2 = \tr F_{h_0}^2
$$
and therefore $dd^c\omega = 0$. From this, $g$ is both conformally balanced and pluriclosed, and hence it must be K\"ahler (\cite{IvPa}), contradicting our assumptions. 
\end{proof}

In a non-K\"ahler manifold $X$ with $h^{0,1}_{\dbar}(X) \neq 0$, the previous result provides continuous families of pairs $(V_0,V)$ for which there cannot be solutions of the Hull-Strominger system with the ansatz \ref{eq:HYMintro}. In particular, one can always make the non-generic choice $V=V_0$, which obstructs the existence of solutions. For the sake of concreteness, we discuss two examples below, which slightly generalize the previous situation.

\begin{example}
Let $(X,\Omega)$ be a compact Calabi-Yau threefold, given by the total space of a non-trivial $T^2$-principal bundle over a  $K3$ surface $S$, as considered in \cite{FuYau,Ga2}. Then, $X$ is non-K\"ahler, balanced, and furthermore has $h^{0,1}_{\dbar}(X) = 1$. By \cite[Lemma 2.1]{Ga2}, there exists a balanced class $\mathfrak{b}$ such that the moduli space of $\mathfrak{b}$-polystable holomorphic structures $V_0$ on $T^{1,0}$ is non-empty. Then, for any choice of such $V_0$ and holomorphic line bundles $L_0, \ldots, L_r$ on $X$ with $c_1(L_j) = 0$, we define
$$
V = V_0 \otimes L_0 \bigoplus (\oplus_{j=1}^r L_j).
$$
Arguing now as in the proof of Proposition \ref{prop:V0h01}, it follows that $(X,\Omega,V)$ does not admit a solution of the system \eqref{eq:HSintro} with balanced class $\mathfrak{b}$, such that $\nabla$ is the Chern connection of a Hermitian-Einstein metric $h_0$ on $V_0$.
\end{example}

Our next example considers the existence problem for the Hull-Strominger system on nilmanifolds, as in the seminal paper \cite{FeIvUgVi}.

\begin{example}\label{ex:noHSNil}
Let $(X,J)$ be a non-K\"ahler compact balanced nilmanifold of complex dimension $3$ with left-invariant complex structure $J$ and trivial canonical bundle. Recall that $X$ is diffeomorphic to a quotient $\Gamma\backslash G$, of a simply-connected nilpotent Lie group $G$ by a co-compact lattice $\Gamma$ of maximal rank. Then, by the classification in \cite[Proposition 2.3]{UgVi}, it follows that  $h^{0,1}_{\dbar}(X) \neq 0$. The smooth tangent bundle is trivial, and we take the holomorphic structure on $T^{1,0} \cong X \times \CC^3$ to be a direct sum of holomorphic line bundles 
$$
V_0 = L_1^0 \oplus L_2^0 \oplus L_3^0
$$
with $c_1(L_j^0) = 0$. Consider the rank-$r$ holomorphic vector bundle
$$
V= \oplus_{j=1}^r L_j,
$$
with $L_j$ holomorphic line bundles with $c_1(L_j) = 0$. Then, $V_0$ and $V$ are both $\mathfrak{b}$-polystable with respect to any balanced class $\mathfrak{b} \in H^{2,2}_{BC}(X,\RR)$. Furthermore, the Hermitian-Einstein metrics with respect to any balanced metric are flat. Arguing now as in the proof of Proposition \ref{prop:V0h01}, it follows that, for any given $\mathfrak{b}$, $(X,\Omega,V)$ does not admit a solution of the Hull-Strominger system \eqref{eq:HSintro} with balanced class $\mathfrak{b}$, such that $\nabla$ is the Chern connection of a Hermitian-Einstein metric $h_0$ on $V_0$.
\end{example}

One can consider other non-generic choices of holomorphic vector bundles which do not admit solutions of the equations, as for instance $V = V_0^*$ or $V = W \oplus W'$ and $V_0 = W^* \oplus W'$, for some choice of polystable bundles $W$ and $W'$ on $X$. We consider an interesting explicit situation in the next example. 

\begin{example} Let $(X,\Omega)$ be the Calabi-Yau compact threefold given by $X=SL(2,\mathbb{Z}[i])\backslash SL(2,\mathbb{C})$. This complex manifold admits a global frame of $T^*_{1,0}$ induced by left-invariant forms on $SL(2,\mathbb{C})$, which satisfy the structure equations
$$
d\omega_1=\omega_{23}, \; \; d\omega_{2}=-\omega_{13}, \; \; d\omega_{3}=\omega_{12}
$$
and such that $\Omega=\omega_{123}$. Explicitly, this frame is dual to the frame given by left-translation of the elements in $T_{[1]}X=\mathfrak{sl}(2,\mathbb{C})$:
$$
X_1=\left(
\begin{array}{c c}
0 & i/2\\
i/2 & 0
\end{array}
\right), \; \; X_2=\left(
\begin{array}{c c}
0 & 1/2\\
-1/2 & 0
\end{array}
\right), \; \; X_3=\left(\begin{array}{c c}
i/2 & 0\\
1 & -i/2
\end{array}
\right).
$$
We fix the balanced class $\mathfrak{b}$ of the hermitian metric
$$
\omega_0=\frac{i}{2}(\omega_{1\overline{1}}+\omega_{2\overline{2}}+\omega_{3\overline{3}}).
$$
Let $W$ be the holomorphic vector bundle on $X\times \mathbb{C}^2$ with Dolbeault operator given by
$$
\overline{\partial}_{W}=\overline{\partial}+\sum_{i=1}^3\omega_{\overline{i}}\otimes X_i
$$ 
and let $V=W^{\oplus 4}$. The integrability $\overline{\partial}_{W}^2 = 0$ boils down to the fact that we are using the standard representation of $\mathfrak{sl}(2,\mathbb{C})$ above for the matrix-valued $(0,1)$-forms of the operator.

Similarly, let $V_0$ be the holomorphic vector bundle on $X\times \mathfrak{sl}(2,\mathbb{C})$ with Dolbeault operator induced by the adjoint representation of $\mathfrak{sl}(2,\mathbb{C})$:
$$
\overline{\partial}_{V_0}=\overline{\partial}+\sum_{i=1}^3 \omega_{\overline{i}}\otimes [X_i,\cdot].
$$
By definition, $V_0$ is the associated bundle to the $SL(2,\mathbb{C})$-principal bundle %$\mathrm{Fr}_1 \ W$ 
of frames of $W$, via the adjoint representation
$$
\rho : SL(2,\mathbb{C}) \rightarrow GL(\mathfrak{sl}(2,\mathbb{C})), \; \; g \mapsto \mathrm{Ad}_g.
$$
In other words, $V_0\cong \mathrm{End}_0 \ W$, where $\mathrm{End}_0$ stands for null-trace endomorphisms. Interestingly, one can prove that $V_0$ is isomorphic to $(T^{1,0},(\nabla^B)^{0,1})$, where $\nabla^B$ denotes the Bismut connection of $\omega_0$ (cf. \cite{FeYa}).

Next, we observe that the standard Hermitian metric on $\mathbb{C}^2$ produces a Hermitian metric $h_W$ on $W$ whose Chern curvature is given by 
$$
F_{h_W}=\left(\begin{array}{c c}
-\tfrac{i}{2}(\omega_{1\overline{2}}-\omega_{2\overline{1}}) & \tfrac{1}{2}((\omega_{1\overline{3}}-\omega_{3\overline{1}})-i(\omega_{2\overline{3}}-\omega_{3\overline{2}}))\\
-\tfrac{1}{2}((\omega_{1\overline{3}}-\omega_{3\overline{1}})+i(\omega_{2\overline{3}}-\omega_{3\overline{2}})) & \tfrac{i}{2}(\omega_{1\overline{2}}-\omega_{2\overline{1}})
\end{array}\right).
$$
It is a straightforward computation to check that
$$
F_{h_W}\wedge \omega_0^2=0.
$$
Therefore, the bundle $W$ is $\mathfrak{b}$-polystable, and so are $V_0$ and $V$.

We now consider the Hull-Strominger system \eqref{eq:HSintro} with the ansatz \eqref{eq:HYMintro} for the bundles $V_0$ and $V$ as above. Suppose $(\omega',h_0,h)$ is a solution with balanced class $\mathfrak{b}$ (with $\nabla$ the Chern connection of $h_0$ on $V_0$). Let $\tilde h_W$ be a Hermitian-Einstein metric on $W$ with respect to $\omega'$. By uniqueness of Hermitian-Einstein metrics, there is a holomorphic gauge transformation $u\in \mathrm{Aut}(V)$ such that $u h = \tilde h_W^{\oplus 4}$, and therefore $\mathrm{tr}_V F_{h}^2 = 4\mathrm{tr}_W F_{\tilde h_W}^2$. Similarly, the Chern connection of $h_0$ is related via a holomorphic gauge tranformation to the connection $\rho_*D^{\tilde h_W}$, induced by the Chern connection $D^{\tilde h_W}$ of $\tilde h_W$ via $\rho$. Then, we have
$$
\mathrm{tr}_{V_0} F_{h_0}^2 = \mathrm{tr}_{V_0}(\rho\circ F_{\tilde h_W})^2=4\mathrm{tr}_W F_{\tilde h_W}^2 
$$
where the last step boils down to checking:
$$
\mathrm{tr}_{\mathfrak{sl}(2,\mathbb{C})}(\rho(A)^2)=4\mathrm{tr}_{\mathbb{C}^2}(A^2), \; \; A\in SL(2,\mathbb{C}).
$$
Therefore, arguing as in the previous examples, we conclude that $\omega'$ is K\"ahler, reaching a contradiction since $X$ admits no K\"ahler metrics. %We conclude that $(X,\Omega,V)$ does not have solutions to the Hull-Strominger system for the (non-generic) choice of $V_0$.
\end{example}

%\begin{remark}\label{rem:category}
%We speculate that the previous obstructions arise from `categorical symmetries' of the equation \eqref{eq:HSintro} with the ansatz \eqref{eq:HYMintro}, which should find an explanation via higher gauge theory \cite{AGaT}.
%\end{remark}

We finish this section with an example that illustrates a different potential obstruction to the existence of solutions for non-K\"ahler manifolds, related to the positivity of the solutions of the Bianchi identity. In fact, for this example we are not able to decide whether there exists a solution of the Hull-Strominger system with the ansatz \eqref{eq:HYMintro}, and speculate that it may yield a negative asnwer to Question \ref{question}.

\begin{example}\label{ex:h-19}
We go back to the situation of Example \ref{ex:noHSNil}, for a compact Calabi-Yau nilmanifold $(X,\Omega)$ with underlying nilpotent Lie algebra $\mathfrak{h}^-_{19}$, considered in \cite[Section 8]{FeIvUgVi}. This complex manifold admits a global frame of $T^*_{1,0}$ induced by left-invariant forms, satisfying the structure equations
$$
d\omega_1=0, \; \; d\omega_2=\omega_{13}+\omega_{1\overline{3}}, \; \; d\omega_3=i(\omega_{1\overline{2}}-\omega_{2\overline{1}}),
$$
and such that $\Omega=\omega_{123}$. The most general $d$-closed, purely imaginary, $(1,1)$-form on $X$ induced by left-invariant forms is given by
$$
F=\pi(m\omega_{1\overline{1}}+ni(\omega_{1\overline{2}}-\omega_{2\overline{1}})),
$$
for $m,n \in \RR$. For a suitable choice of lattice, one can show that, for any $(m,n)\in \mathbb{Z}^{2}$, $\tfrac{i}{2\pi} F$ has integral periods and hence, by general theory, this is the curvature form of the Chern connection of a holomorphic hermitian line bundle $(L,h)\rightarrow X$. 

We fix the balanced Hermitian form
$$
\omega=\tfrac{i}{2}(\omega_{1\overline{1}}+\omega_{2\overline{2}}+\omega_{3\overline{3}})
$$
and consider the associated balanced class $\mathfrak{b}$. With the previous notation, one can easily see that $c_1(L) \cdot \mathfrak{b} = 0$ when $m=0$. Hence, for a choice of integers $n_j \in \mathbb{Z}\backslash \{0\}$, $j = 1, \ldots, r$, with associated line bundles $L_j$ as above, the holomorphic vector 
bundle
$$
V= \bigoplus_{j=1}^r L_j
$$
is polystable with respect to $\mathfrak{b}$.

We take a holomorphic structure $V_0$ on $T^{1,0}$, with Dolbeault operator of the form 
$$
\overline{\partial}^\lambda_{\overline{X_i}} X_j=\sum_{k=1}^3 \lambda_{ijk}X_k,
$$
where $\{X_i\}_{i=1,2,3}$ is the dual frame of $\{\omega_i\}_{i=1,2,3}$, and $\lambda_{i,j,k}$ are constant complex functions. Assuming that $V_0$ is polystable with respect to $\mathfrak{b}$, one can prove that
$$
V_0\cong L_1^0\oplus L_2^0\oplus L_3^0
$$
for $L_i^0$ line bundles with $c_1(L_i^0)=0$.

With this setup, we consider the Hull-Strominger system \eqref{eq:HSintro} with coupling constant $\alpha \in \RR$ and the ansantz \eqref{eq:HYMintro}, that is, for triples $(\omega',\oplus_{j=1}^3 \tilde h_j^0,\oplus_{j=1}^r \tilde h_j)$ with $\omega'$ a Hermitian form on $X$ and $\tilde h_j^0$ (resp. $\tilde h_j$) a Hermitian metric on $L^0_j$ (resp. $L_j$). Provided that we have a solution, it is clear, in particular, that the Hermitian-Einstein metrics $\tilde{h_j^0}$ with respect to $\omega'$ are flat. Assume first that $\alpha > 0$. Then, by \cite{CHSW} (see also \cite[Proposition 2.14]{GFGM}), any solution must satisfy $d \omega' = 0$ and $F_{\tilde{h}_j} = 0$, in contradiction with our assumptions. In the case $\alpha=0$, any solution is again K\"ahler by \cite{IvPa}.

In the remaining case of $\alpha<0$, we assume that our triples $(\omega',\oplus_{j=1}^3 \tilde h_j^0,\oplus_{j=1}^r \tilde h_j)$ are such that $\omega'$, $F_{\tilde h_j}$ and $F_{\tilde h_j^0}$ are invariant $(1,1)$ forms on $X$. Then, it follows that $F_{\tilde h_j} = F_{h_j}$ and $F_{\tilde h_j^0} = 0$, and hence the Bianchi identity reduces to
$$
dd^c\omega' = - \alpha \sum_{j=1}^r F_{h_j}^2 = -2\alpha\pi^2\sum_{j=1}^r n_j^2\omega_{12\overline{12}}.
$$
One can prove that the general solution of the previous equation is given by
\begin{align*}
\omega' = -\frac{\alpha\pi^2}{2}(& \sum_{j=1}^r
n_j^2 \omega_{3\overline{3}} + s_1 i \omega_{1\overline{1}}+s_2(\omega_{1\overline{2}}-\omega_{2\overline{1}})+s_3 i(\omega_{1\overline{2}}+\omega_{2\overline{1}})+\\
& +s_4(\omega_{1\overline{3}}-\omega_{3\overline{1}})+s_5 i(\omega_{1\overline{3}}+\omega_{3\overline{1}})+s_6(\omega_{2\overline{3}}-\omega_{3\overline{2}})+s_7 i(\omega_{3\overline{2}}+\omega_{2\overline{1}}))
\end{align*}
where $s_i$ are real constants and hence, since the component in $\omega_{2\overline{2}}$ vanishes, it follows that $\omega'$ is necessarily non-positive. This proves that, among the invariant solutions of the Bianchi identity for this choice of $(V,V_0)$, there are no potential solutions of the Hull-Strominger system \eqref{eq:HSintro}, since the corresponding $\omega'$ is not a Hermitian metric.
\end{example}

\section{Background on holomorphic Courant algebroids}\label{sec:generalizedHerm}

\subsection{Holomorphic Courant algebroids}

Let $X$ be a complex manifold. We denote by $\mathcal{O}_X$ and $\underline{\mathbb{C}}$ the sheaves of holomorphic functions and $\mathbb{C}$-valued locally constant functions on $X$, respectively. We denote by $T^{1,0}$ and $T_{1,0}^*$ the holomorphic tangent and cotangent bundles of $X$, respectively.

\begin{definition}\label{d:CAhol} 
A \emph{holomorphic Courant algebroid} is a holomorphic vector bundle $\mathcal{Q} \to X$ together with a nondegenerate holomorphic symmetric bilinear form $\IP{,}$, a holomorphic vector bundle morphism $\pi:\mathcal{Q}\to T^{1,0}$ called anchor map, and a Dorfman bracket on holomorphic sections of $\mathcal{Q}$, that is, a homomorphism of sheaves of $\underline{\mathbb{C}}$-modules
$$
[ \cdot,\cdot ] \colon \mathcal{Q} \otimes_{\underline{\mathbb{C}}} \mathcal{Q} \to \mathcal{Q},
$$
satisfying, for $u,v,w\in \mathcal{Q}$ and $\phi\in \mathcal{O}_X$,
  \begin{itemize}
  \item[(D1):] $[u,[v,w]] = [[u,v],w] + [v,[u,w]]$,
  \item[(D2):] $\pi([u,v])=[\pi(u),\pi(v)]$,
  \item[(D3):] $[u,\phi v] = \pi(u)(\phi) v + \phi[u,v]$,
  \item[(D4):] $\pi(u)\IP{v,w} = \IP{[u,v],w} + \IP{v,[u,w]}$,
  \item[(D5):] $[u,v]+[v,u]= \mathcal{D}\IP{u,v}$,
  \end{itemize}
where $\mathcal{D} \colon \mathcal{O}_X \to \mathcal{Q}$ denotes the composition of the exterior differential, the natural map $\pi^* \colon T^*X \to \mathcal{Q}^*$, and the isomorphism $\mathcal{Q}^* \to\mathcal{Q}$ provided by $\IP{\cdot,\cdot}$.
\end{definition}

We will denote a holomorphic Courant algebroid $(\mathcal{Q},\IP{,},[\cdot,\cdot],\pi)$ simply by $\mathcal{Q}$. Using the isomorphism $\IP{,} \colon \mathcal{Q} \to \mathcal{Q}^*$ we obtain a complex of holomorphic vector bundles
\begin{equation}\label{eq:holCouseqaux}
T^*_{1,0} \overset{\pi^*}{\longrightarrow} \mathcal{Q} \overset{\pi}{\longrightarrow} T^{1,0}.
\end{equation}

\begin{definition}\label{d:CAholexact} 
We will say that $\mathcal{Q}$ is \emph{transitive} if the anchor map $\pi$ in \eqref{eq:holCouseqaux} is surjective.
\end{definition}

Given a transitive holomorphic Courant algebroid $\mathcal{Q}$ over $X$, there is an associated holomorphic Lie algebroid 
$$
A_\mathcal{Q} := \mathcal{Q}/(\Ker \pi)^\perp.
$$ 
Furthermore, the holomorphic subbundle
$$
\ad_\mathcal{Q} := \Ker \pi/(\Ker \pi)^\perp \subset A_\mathcal{Q}
$$
inherits the structure of a holomorphic bundle of quadratic Lie algebras. Therefore, the bundle $\mathcal{Q}$ fits into a double extension of holomorphic vector bundles
\begin{equation}\label{eq:holCoustr}
\begin{split}
0 \longrightarrow T^*_{1,0} \overset{\pi^*}{\longrightarrow} \mathcal{Q} \longrightarrow A_\mathcal{Q} \longrightarrow 0\\
0 \longrightarrow \ad_\mathcal{Q} \overset{\pi^*}{\longrightarrow} A_\mathcal{Q} \overset{\pi}{\longrightarrow} T^{1,0} \longrightarrow 0.
\end{split}
\end{equation}
A classification of transitive holomorphic Courant algebroids has been obtained in \cite[Proposition 2.9]{GaRuTi2} in the special case that $A_\mathcal{Q}$ is isomorphic to the Atiyah algebroid of a holomorphic principal bundle. A holomorphic Courant algebroid of this form is said  to be of \emph{string type}. Rather than explaining this complicated classification, which we will not directly use here, we briefly discuss the basic example which follows from it.

\begin{example}\label{def:Q0}
Let $G$ be a complex Lie group with Lie algebra $\mathfrak{g}$. We assume that $\mathfrak{g}$ is endowed with a non-degenerate bi-invariant symmetric bilinear form
$$
\IP{,}: \mathfrak{g} \otimes \mathfrak{g} \to \mathbb{C}.
$$
Let $p \colon P \to X$ be a holomorphic principal $G$-bundle over $X$.
% with vanishing first Pontryagin class
%$$
%p_1(P) = 0 \in H^4(X,\CC).
%$$
Consider the holomorphic Atiyah Lie algebroid $A_{P} := T^{1,0}P/G$ of $P$, with  bracket induced by the Lie bracket on $T^{1,0}P$. The holomorphic bundle of Lie algebras 
$
\ad P:= \Ker dp \subset A_{P}
$
fits into the short exact sequence of holomorphic Lie algebroids
$$
0 \to \ad P \to A_{P} \to T^{1,0} \to 0.
$$
We construct next a transitive holomorphic Courant algebroid such that the second sequence in \eqref{eq:holCoustr} is canonically isomorphic to the exact sequence of holomorphic Lie algebroids above. For this, we assume the existence of a pair $(\tau,\theta)$ as follows: $\theta$ is a principal connection on the smooth $G$-bundle underlying $P$ whose curvature $F_\theta$ satisfies 
$$
F_\theta^{0,2} = 0
$$
and whose $(0,1)$-part induces the holomorphic structure on $P$, while $\tau \in \Omega^{3,0} \oplus \Omega^{2,1}$ is a complex three-form such that
\begin{equation}\label{eq:Bianchihol}
d\tau - \IP{F_\theta \wedge F_\theta} = 0.
\end{equation}
Here, $\Omega^{p,q}$ denotes the space of smooth complex forms on $X$ of type $(p,q)$. Given such a pair $(\tau,\theta)$, we define a holomorphic Courant algebroid $\mathcal{Q}_{P,\tau,\theta}$ with underlying smooth complex vector bundle
\begin{equation*}%\label{eq:Qexpb}
T^{1,0} \oplus\ad P\oplus T_{1,0}^*,
\end{equation*}
with Dolbeault operator
\begin{equation}\label{eq:DolQ}
\dbar_0 (V + r + \xi)  = \dbar V + i_V F_\theta^{1,1} + \dbar^\theta r + \dbar \xi - i_{V}\tau^{2,1} + 2\IP{ F_\theta^{1,1}, r}
\end{equation}
non-degenerate symmetric bilinear form
$$
\IP{V + r + \xi , V + r + \xi}_0  = \xi(V) + \IP{r,r},
$$
bracket,
\begin{equation*}
	\begin{split}
	[V+ r + \xi,W + t + \eta]_0   = {} & [V,W] - F^{2,0}_\theta(V,W) + \partial^\theta_V t - \partial^\theta_W r - [r,t]\\
	& {} + i_V \partial \eta + \partial (\eta(V)) - i_W\partial \xi + i_Wi_V \tau^{3,0},\\
	& {} + 2\IP{\partial^\theta r, t} + 2\IP{i_V F_\theta^{2,0}, t} - 2\IP{i_W F_\theta^{2,0}, r},	
	\end{split}
\end{equation*}
and anchor map $\pi_0(V+ r + \xi) = V$. It is an exercise to show that $\mathcal{Q}_{P,\tau,\theta}$ defines a transitive holomorphic Courant algebroid in the sense of Definition \ref{d:CAholexact} (see \cite[Proposition 2.4]{GaRuTi2}).
\end{example}

\begin{remark}
Notice that, in particular, condition \eqref{eq:Bianchihol} implies the vanishing of the \emph{first Pontryagin class}
\begin{equation}\label{eq:p1(P)=0dRC}
p_1(P) = 0 \in H^4_{dR}(X,\mathbb{C}),
\end{equation}
where $p_1(P)$ denotes the invariant associated to the pairing $\IP{,}$ on $\mathfrak{g}$ via Chern-Weyl Theory. When the complex manifold satisfies the $\partial \dbar$-Lemma, the condition \eqref{eq:p1(P)=0dRC} is in fact sufficient for the construction of a transitive holomorphic Courant algebroid of string type, as in Example \ref{def:Q0} (see \cite[Corollary 3.6]{GaRuTi2}). % In this setup, it can be proved that, given $P$, the associated $\mathcal{Q}$ is unique up to isomorphism (see \cite[Corollary 3.6]{GaRuTi2}).
\end{remark}

\subsection{Reduction}\label{s:red}

The holomorphic Courant algebroids of our main interest arise via a reduction mechanism from a special type of Courant algebroids in the smooth category. Recall that a Courant algebroid over a smooth manifold $M$ is given by a real orthogonal bundle $(E,\IP{,})$, a bracket $[,]$ on smooth sections, and a bracket preserving morphism $\pi \colon E \to T:= TM$ satisfying a set of axioms analogous to those in Definition \ref{d:CAhol}. In the sequel, a Courant algebroid over $M$ will be denoted simply by $E$.
 
Let $E$ be a smooth Courant algebroid over a smooth manifold $M$. We assume that $M$ is endowed with an integrable complex structure $J$. Consider the smooth complex Courant algebroid $ E \otimes \mathbb{C}$, with Courant structure given by the $\mathbb{C}$-linear extensions of the symmetric bilinear form  $\IP{,}$, the bracket $[\cdot, \cdot]$, and the anchor map $\pi$. We recall next the notion of lifting, which will enable us to construct a holomorphic Courant algebroid out of $E \otimes \mathbb{C}$.

\begin{definition}\label{def:lifting}
Let $E$ be a smooth Courant algebroid over a complex manifold $X = (M,J)$. A \emph{lifting} of $T^{0,1}$ to $E \otimes \mathbb{C}$ is an isotropic, involutive subbundle $\ell \subset E \otimes \mathbb{C}$ mapping isomorphically to $T^{0,1}$ under the $\mathbb{C}$-linear extension of the anchor map $\pi \colon E \otimes \mathbb{C} \to T \otimes \mathbb{C}$.
\end{definition}

Given a lifting $\ell$ of $T^{0,1}$ to $E \otimes \mathbb{C}$, following \cite{GualtieriGKG} we consider the reduction of $E \otimes \mathbb{C}$ by $\ell$ given by the orthogonal bundle 
$$
\mathcal{Q}_\ell := \ell^\perp/\ell,
$$
where $\ell^\perp$ is the orthogonal complement of $\ell$ with respect to the symmetric pairing on $E \otimes \mathbb{C}$. Since $\ell$ is a lifting of $T^{0,1}$, the kernel of $\pi_{|\ell^\perp}$ contains $T^*_{1,0}$, and therefore $\mathcal{Q}_\ell$ fits in a vector bundle complex of the form \eqref{eq:holCouseqaux}. The Dolbeault operator on $\mathcal{Q}_\ell$ is defined as follows: given $s$ a smooth section of $\mathcal{Q}_\ell$, we define
$$
\overline{\partial}^{\ell}_V s = [\tilde V,\tilde s] \quad \textrm{mod}\ \ell
$$
where $V \in T^{0,1}$, $\tilde V$ is the unique lift of $V$ to $\ell$, and $\tilde s$ is any lift of $s$ to a section of $\ell^\perp$. The Jacobi identity for the Dorfman bracket on $E \otimes \mathbb{C}$ implies that $\overline{\partial}^{\ell} \circ \overline{\partial}^{\ell} = 0$ and that it
induces a Dorfman bracket on the holomorphic sections of $\mathcal{Q}_\ell$. It is not difficult to prove that $\mathcal{Q}_\ell$ defines a holomorphic Courant algebroid in the sense of Definition \ref{d:CAholexact} (cf. \cite[Proposition 2.8]{GaRuTi3}).
 
In the sequel we will focus %on a class of holomorphic Courant algebroids obtained by reduction of a 
on a special type of smooth transitive Courant algebroids, called \emph{string}. This will help us to make the construction of $\mathcal{Q}_\ell$ above more explicit, recovering the construction in Example \ref{def:Q0}. Recall that a Courant algebroid $E$ is called transitive if the anchor map is surjective. In particular, such an object defines a double extension of smooth real vector bundles
\begin{equation}\label{eq:smoothCoustr}
\begin{split}
0 \longrightarrow T^* \overset{\pi^*}{\longrightarrow} E {\longrightarrow} A_E \longrightarrow 0,\\
0 \longrightarrow \ad_E {\longrightarrow} A_E {\longrightarrow} T \longrightarrow 0.
\end{split}
\end{equation}
Here, the Lie algebroids $A_E := E/(\Ker \pi)^\perp$ and $\ad_E = \Ker \pi/(\Ker \pi)^\perp$ are defined as in the holomorphic case. The basic example which we will need is as follows.

\begin{example}\label{def:E0}
Let $K$ be a real Lie group with Lie algebra $\mathfrak{k}$. We assume that $\mathfrak{k}$ is endowed with a non-degenerate bi-invariant symmetric bilinear form
$$
\IP{,}: \mathfrak{k} \otimes \mathfrak{k} \to \mathbb{R}.
$$
Let $p \colon P_K \to M$ be a smooth principal $K$-bundle over $M$. Consider the Atiyah Lie algebroid $A_{P_K} := TP_K/K$. The smooth bundle of Lie algebras 
$
\ad P_K:= \Ker dp \subset A_{P_K}
$
fits into the short exact sequence of Lie algebroids
$$
0 \to \ad P_K \to A_{P_K} \to T \to 0.
$$
We construct next a transitive Courant algebroid such that the second sequence in \eqref{eq:holCoustr} is canonically isomorphic to the exact sequence of Lie algebroids above. For this, we assume that
$$
p_1(P_K) = 0 \in H^4_{dR}(M,\mathbb{R}),
$$
where $p_1(P)$ denotes \emph{first Pontryagin class} of $P_K$ associated to the bi-invariant pairing $\IP{,}$ on $\mathfrak{g}$ via Chern-Weyl Theory. Then, given a choice of principal connection $\theta$ on $P_K$ there exists a smooth real three-form $H \in \Omega^{3}$ such that
\begin{equation}\label{eq:Bianchireal}
dH - \IP{F_\theta \wedge F_\theta} = 0.
\end{equation}
Given such a pair $(H,\theta)$, we define a smooth Courant algebroid $E_{P_K,H,\theta}$ with underlying vector bundle
\begin{equation*}%\label{eq:Qexpb}
T \oplus\ad P_K\oplus T^*,
\end{equation*}
non-degenerate symmetric bilinear form $\IP{,}_0$ and anchor map $\pi_0$ defined as in Example \ref{def:Q0}, and bracket given by
%$$
%\IP{V + r + \xi , V + r + \xi}_0  = \xi(V) + \IP{r,r},
%$$
\begin{equation*}
	\begin{split}
	[V+ r + \xi,W + t + \eta]_0   = {} & [V,W] - F_\theta(V,W) + d^\theta_V t - d^\theta_W r - [r,t]\\
	& {} + L_V \eta - i_W d \xi + i_Wi_V H\\
	& {} + 2\IP{d^\theta r, t} + 2\IP{i_V F_\theta, t} - 2\IP{i_W F_\theta, r}.
	\end{split}
\end{equation*}
%and anchor map $\pi_0(V+ r + \xi) = V$.
It is not difficult to see that $E_{P_K,H,\theta}$ defines a smooth transitive Courant algebroid over $M$, as defined above.
\end{example}

Transitive Courant algebroids as in Example \ref{def:E0} fit into the category of \email{smooth string algebroids} \cite{GaRuTi2}, which motivates the following definition.

\begin{definition}\label{d:CAstring} 
A smooth Courant algebroid $E$ over $M$ is of \emph{string type} if it is isomorphic to a Courant algebroid $E_{P_K,H,\theta}$ as in Example \ref{def:E0}, for some triple $(P_K,H,\theta)$ satisfying \eqref{eq:Bianchireal}.
\end{definition}

The notion of isomorphism which we use here is the standard one for smooth Courant algebroids, given by orthogonal bundle morphisms which preserve the bracket and the anchor map (cf. \cite[Definition 2.3]{GaRuTi2}).

The following result characterizes explicitly the liftings on a smooth Courant algebroid of string type. It follows easily from \cite[Lemma 2.15]{GaRuTi3} combined with the \emph{Chern correspondence} in \cite[Lemma 5.11]{GaRuTi3}, and hence we omit the proof. For the statement, we denote the space of real $(p,p)$-forms on a complex manifold by $\Omega^{p,p}_\RR$. Recall that, given a Courant algebroid $E$ as in Example \ref{def:E0}, a pair $(\gamma,\beta) \in \Omega^{2}_\CC \oplus \Omega^{1}_\CC(\ad \underline P_K)$ induces an orthogonal automorphism of $E \otimes \CC$  by
\begin{equation}\label{eq:BA}
(\gamma,\beta)(V + r + \xi) = V + i_V \beta + r + i_V \gamma - \IP{i_V\beta,\beta} - 2\IP{\beta,r} + \xi.
\end{equation}

\begin{lemma}[\cite{GaRuTi3}]\label{lemma:liftings}
Let $X$ be a complex manifold. Let $E_0 = E_{P_K,H_0,\theta_0}$ be the smooth Courant algebroid of string type in Example \ref{def:E0}. Then, a lifting $\ell \subset E_0 \otimes \mathbb{C}$  of $T^{0,1}$ is equivalent to a triple
	$$
	(\omega,b,a) \in \Omega^{1,1}_\RR \oplus \Omega^2 \oplus \Omega^{1}(\ad P_K)
	$$
	satisfying 
	\begin{equation}\label{eq:liftingcond}
	\begin{split}
	H_0 + d^c \omega - d b + 2 \IP{a\wedge F_{\theta_0}} + \IP{a\wedge d^{\theta_0} a } + \frac{1}{3} \IP{a\wedge [a\wedge a]} & = 0,\\
	F_{\theta_0}^{0,2} + \dbar^{\theta_0}a^{0,1} + \frac{1}{2}[a^{0,1}\wedge a^{0,1}] & = 0.\\
	\end{split}
	\end{equation}
More precisely, given $(\omega,b,a)$ satisfying \eqref{eq:liftingcond}, the lifting is
	\begin{equation}\label{eq:L}
	\ell(\omega,b,a) = \{(i\omega - b,-a)(V^{0,1}), \; V^{0,1} \in T^{0,1} \},
	\end{equation}
and, conversely, any lifting is uniquely expressed in this way. In particular, setting $\theta = \theta_0 + a$, equation \eqref{eq:liftingcond} implies that
\begin{equation}\label{eq:anomaly}
dd^c\omega + \IP{F_\theta \wedge F_\theta} = 0, \qquad F_\theta^{0,2} = 0.
\end{equation}
\end{lemma}

The next result, which follows from \cite[Proposition 2.16]{GaRuTi2}, gives an explicit formula for the holomorphic Courant algebroid associated to a lifting $\ell(\omega,b,a)$. For this, we use the identity
$$
\ell(\omega,b,a)^\perp = \ell(\omega,b,a) \oplus (i\omega -b,-a)(T^{1,0}  \oplus(\mathrm{ad} \underline{P}_K)\otimes \mathbb{C} \oplus T^*_{1,0}).
$$

\begin{lemma}\label{l:liftingQ} 
Let $X$ be a complex manifold endowed with a smooth Courant algebroid of string type $E_{P_K,H_0,\theta_0}$ as above, with $K$ compact. Let $\ell(\omega,b,a)$ be a lifting of $T^{0,1}$ as in Lemma \ref{lemma:liftings}. Then, using the notation in Example \ref{def:Q0}, there is a canonical isomorphism
$$
\mathcal{Q}_{\ell(\omega,b,a)} \cong \mathcal{Q}_{P,2 i\partial \omega,\theta},
$$
given by
$$
[(i\omega-b,-a)(X+r+\xi^{1,0})]\mapsto X^{1,0}+r+\xi^{1,0},
$$
where $P = (P_K \times_K G,\theta^{0,1})$ is the holomorphic principal $G$-bundle induced by $\theta$, for $G = K^c$ the complexification of $K$.
\end{lemma}

\subsection{Bott-Chern algebroids}\label{sec:BC}

The holomorphic Courant algebroids obtained in Lemma \ref{l:liftingQ} correspond to a special class of those considered in Example \ref{def:Q0}. For instance, if $\mathcal{Q}$ is obtained by reduction from a smooth Courant algebroid of string type, then it must be isomorphic to $\mathcal{Q}_{P,\tau,\theta}$, where $\tau = 2 i\partial \omega$ for some $\omega \in \Omega^{1,1}_\RR$, and $\theta$ the Chern connection of a reduction of $P$ to a maximal compact subgroup. Following \cite{GaRuShTi,GaRuTi3}, we introduce the following definition.

\begin{definition}\label{d:CABC} 
A holomorphic Courant algebroid $\mathcal{Q}$ over a complex manifold $X$ is of \emph{Bott-Chern type} if it is isomorphic to the reduction $\mathcal{Q}_\ell$ of a smooth Courant algebroid of string type $E$, for some lifting $\ell$ of $T^{0,1}$ to $E \otimes \CC$.
\end{definition}

We recall next the classification of holomorphic Courant algebroids of Bott-Chern type with fixed holomorphic principal bundle \cite{GaRuShTi}. Let $G$ be a complex reductive Lie group with Lie algebra $\mathfrak{g}$. Fix a non-degenerate bi-invariant symmetric bilinear form $\IP{,}: \mathfrak{g} \otimes \mathfrak{g} \to \mathbb{C}$ such that $\IP{\mathfrak{k},\mathfrak{k}} \subset \RR$, for any compact Lie subalgebra $\mathfrak{k} \subset \mathfrak{g}$. Let $P$ be a holomorphic principal $G$-bundle over a compact complex manifold $X$ such that
\begin{equation}\label{eq:BC22}
p_1(P) = 0 \in H^{2,2}_{BC}(X,\RR),
\end{equation}
where $p_1$ is defined as in \eqref{eq:p1(P)=0dRC}. Then, for any choice of reduction $h$ of $P$ to a maximal compact subgroup $K \subset G$, with Chern connection $\theta^h$, we can choose $\omega \in \Omega^{1,1}_\RR$ such that
$$
dd^c\omega + \IP{F_h \wedge F_h} = 0,
$$
and construct a Bott-Chern algebroid $\mathcal{Q} := \mathcal{Q}_{P,2i\partial\omega,\theta^h}$ as in Example \ref{def:Q0}. By construction, $\mathcal{Q}$ fits in a holomorphic extension of the form
\begin{equation}\label{eq:defstring}
	\xymatrix{
		0 \ar[r] & T^*_{1,0} \ar[r] & \mathcal{Q} \ar[r]^\rho & A_P \ar[r] & 0,
	}
	\end{equation}
where $\rho(V + r + \xi) = \theta^h V + r$ is a bracket preserving map ($\theta^h V$ denoting the horizontal lift) inducing a Lie algebroid isomorphism $A_\mathcal{Q} \cong A_P$ and a quadratic Lie algebra bundle isomorphism  $\ad_\mathcal{Q} \cong (\ad P,\IP{,})$. The classification of our interest is for pairs $(\mathcal{Q},\rho)$ as before (which we call Bott-Chern algebroids, following \cite{GaRuShTi}), via conmutative diagrams, as follows:
\begin{equation}\label{eq:defstringiso}
    \xymatrix{
      0 \ar[r] &  T^*_{1,0} \ar[r] \ar[d]^{id} & \mathcal{Q} \ar[r]^\rho \ar[d]^{\varphi} & A_P \ar[r] \ar[d]^{id} & 0\\
      0 \ar[r] &  T^*_{1,0} \ar[r] & \mathcal{Q}' \ar[r]^{\rho'} & A_{P} \ar[r] & 0
    }
\end{equation}
where $\varphi \colon \mathcal{Q} \to \mathcal{Q}'$ is an isomorphism of holomorphic Courant algebroids.

Consider the Aeppli cohomology groups of the complex manifold $H^{p,q}_{A}(X)$, defined by
\begin{equation}\label{eq:A}
H^{p,q}_{A}(X)=\frac{\mathrm{ker} \hspace{1mm} \partial\overline{\partial}: \Omega^{p,q}(X,\mathbb{C}) \longrightarrow \Omega^{p+1,q+1}(X,\mathbb{C})}{\mathrm{Im} \hspace{1mm}  \partial \oplus \overline{\partial}: \Omega^{p-1,q}(X,\mathbb{C}) \oplus \Omega^{p,q-1}(X,\mathbb{C}) \longrightarrow \Omega^{p,q}(X,\mathbb{C})}.
\end{equation}
We denote by $H^{p,p}_{A}(X,\RR) \subset H^{p,p}_{A}(X)$ the canonical real structure. Let $\Omega^{2,0}_{cl}$ be the sheaf of closed $(2,0)$-forms on $X$, whose first cohomology can be described as \cite{GualtieriGKG}
\begin{equation*}%\label{eq:H1Omega20}
H^1(\Omega^{2,0}_{cl}) \cong \frac{\Ker \; d \colon \Omega^{3,0} \oplus \Omega^{2,1} \to \Omega^{4,0} \oplus \Omega^{3,1} \oplus \Omega^{2,2}}{ \operatorname{Im} \; d \colon \Omega^{2,0} \to \Omega^{3,0} \oplus \Omega^{2,1}}.
\end{equation*}
Taking representatives, the $\partial$-operator induces a well-defined linear map 
\begin{equation}\label{eq:partialmap}
\partial \colon H^{1,1}_A(X,\mathbb{R}) \to H^1(\Omega^{2,0}_{cl}) .
\end{equation}

\begin{proposition}[\cite{GaRuShTi}]\label{prop:BCclassification}
Let $\mathfrak{S}$ be the set of equivalence classes of Bott-Chern algebroids over $X$ with fixed principal bundle $P$ and bundle of quadratic Lie algebras $(\ad P,\IP{,})$, defined via the diagrams \eqref{eq:defstringiso}. Then, $\mathfrak{S}$ is an affine space for the vector space given by the image of \eqref{eq:partialmap}. In particular, if $X$ is a $\partial\bar{\partial}$-manifold then there is only one equivalence class.
\end{proposition}

\begin{remark}\label{rem:diffisom}
Observe that the notion of isomorphism considered here is stronger than in \cite{GaRuTi2}, where the Courant algebroid isomorphism $\varphi \colon \mathcal{Q} \to \mathcal{Q}'$ can cover $dg \colon A_P \to A_P$ for $g$ a non-trivial holomorphic gauge transformation on $P$.
\end{remark}

\begin{remark}\label{rem:forgetful}
Consider the composition of the map $H^1(\Omega^{2,0}_{cl}) \to H^{2,1}_{\dbar}(X)$, defined by $[\tau] \to [\tau^{2,1}]$, with \eqref{eq:partialmap}. This composition can be interpreted as a forgetful map, whose image classifies holomorphic extensions of the Atiyah-Lie algebroid $A_P$ by $T^*_{1,0}$ underlying Bott-Chern algebroids with fixed principal bundle $P$. This weaker structure is actually what we need for the definition of the Futaki invariants in Section \ref{sec:Futaki}.
\end{remark}

\section{Generalized Hermitian metrics on Bott-Chern algebroids}\label{sec:GHermitianBC}

\subsection{Generalized Hermitian metrics}\label{sec:GHermitian}

We introduce next generalized Hermitian metrics on Bott-Chern algebroids, with possibly indefinite signature. Our discussion follows closely \cite{Ga0,GaJoSt,GaRuTi1}. Recall that a generalized metric on a smooth Courant algebroid $E$ of string type is given by an orthogonal decomposition 
$$
E = V_+ \oplus V_-
$$
satisfying that the restriction of the ambient metric to $V_+$ is positive definite and that $\pi_{|V_+}:V_{+}\rightarrow T$ is an isomorphism.  Recall that a generalized metric determines uniquely a Riemann metric $g$ on $M$ and an isotropic splitting of $E$. In particular, it has an associated isomorphism $E \cong E_{P_K,H,\theta}$ for a uniquely determined three-form $H$, and connection $\theta$ satisfying \eqref{eq:Bianchireal} (see Example \ref{def:E0}). Furthermore, via this identification we have
\begin{equation}\label{eq:Vpm}
V_+ = \{V + g(V), V \in T\}, \quad V_- = \{V - g(V) + r, V \in T, r \in \ad P_K\}.
\end{equation}
The basic interaction between generalized metrics and complex geometry is provided by the following definition.

\begin{definition}\label{d:generalizedmetricomp}
Let $X$ be a complex manifold endowed with a smooth Courant algebroid $E$ of string type. We say that a generalized metric $E = V_+ \oplus V_-$ is \emph{compatible with $J$} if
$$
\ell = \{e \in V_+ \otimes \mathbb{C},\ \pi(e) \in T^{0,1}\} \subset E \otimes \mathbb{C}
$$
is a lifting of $T^{0,1}$.
\end{definition}

Using the splitting of $E$ determined by the generalized metric, it is not difficult to see that Definition \ref{d:generalizedmetricomp} implies that $g$ is Hermitian and furthermore
$$
\ell = e^{i\omega}T^{0,1} \subset E \otimes \CC
$$
where $\omega = g(J,)$ is the associated Hermitian form. Applying now Lemma \ref{lemma:liftings} we obtain the following.

\begin{lemma}\label{l:lifting2}
Let $X = (M,J)$ be a complex manifold endowed with a smooth Courant algebroid of string type $E$. A generalized metric $E = V_+ \oplus V_-$ is compatible with $J$ if and only if the associated Riemannian metric $g$ is Hermitian and furthermore 
\begin{equation}\label{eq:liftingcondmet}
H = -d^c \omega, \qquad F_\theta^{0,2} = 0.
\end{equation}
In particular 
$$
dd^c \omega + \IP{F_\theta \wedge F_\theta} = 0.
$$
\end{lemma}

Given a compatible generalized metric, consider the associated holomorphic Courant algebroid $\mathcal{Q}_\ell \cong \mathcal{Q}_{P,2 i\partial \omega,\theta}$ of Bott-Chern type (see Lemma \ref{l:liftingQ} and Definition \ref{d:CABC}). We find next an alternative presentation of $\mathcal{Q}_\ell$ which will naturally endow this bundle with a Hermitian metric, possibly with indefinite signature. To see this, note that $V_+^\perp = V_-$ implies that
$$
\ell^\perp = (V_- \otimes \mathbb{C}) \oplus \ell.
$$
Therefore, as a smooth orthogonal bundle $\mathcal{Q}_\ell$ is canonically isomorphic to 
$$
\mathcal{Q}_\ell := \ell^\perp/\ell \cong V_- \otimes \mathbb{C}.
$$

\begin{definition}\label{d:generalizedmetriHerm}
Let $X$ be a complex manifold endowed with a smooth Courant algebroid of string type $E$ and a compatible generalized metric $E = V_+ \oplus V_-$. The induced \emph{generalized Hermitian metric} $\mathbf{G}$ on $\mathcal{Q}_\ell$ is defined by
$$
\mathbf{G}([s_1],[s_2]) = - \IP{\pi_- s_1, \overline{\pi_- s_2}}
$$
for $[s_j] \in \ell^\perp/\ell$ and $\pi_- \colon \ell^\perp \to V_- \otimes \mathbb{C}$ the orthogonal projection.
\end{definition}

We are ready to prove the main result of this section, where we calculate the Chern connection of the induced generalized Hermitian metric $\mathbf{G}$ in terms of the underlying pair $(H,\theta)$ (see Lemma \ref{l:lifting2}). Our result extends \emph{Bismut's Identity} (see \cite[Theorem 2.9]{Bismut}), interpreted recently in \cite{GaJoSt} in the language of holomorphic Courant algebroids. 

\begin{proposition}\label{t:Cherngeneralized} 
Let $X = (M,J)$ be a complex manifold endowed with a smooth Courant algebroid of string type $E$ and a compatible generalized metric $E = V_+ \oplus V_-$. Then, via the isomorphism $\mathcal{Q}_\ell \cong V_- \otimes \mathbb{C}$, the Chern connection of the associated generalized Hermitian metric $\mathbf{G}$ on $\mathcal{Q}_\ell$ is given by
\begin{align}\label{eq:Nablaminusbisabs}
D^{\mathbf{G}}_V s = \pi_-[\sigma_+ V,s].
\end{align}
Here, $\sigma_+ V = V + g(V)$ is the inverse of the isomorphism $\pi_{|V_+} \colon V_+ \to T$. More explicitly, via the identification $V_- \cong T \oplus \ad P_K$, we have
\begin{align}\label{eq:NablaChernexp}
D^{\mathbf{G}}_V (W + r) = \nabla^-_V W - g^{-1} \IP{i_V F_\theta,r} + d^\theta_V r - F_\theta(V,W),
\end{align}
where $\nabla^- = \nabla + \tfrac{1}{2} g^{-1}d^c \omega$, for $\nabla$ the Levi-Civita connection of $g$.
\begin{proof}
The right hand side of \eqref{eq:Nablaminusbisabs} defines an orthogonal connection on $V_-$, which extends $\mathbb{C}$-linearly to a $\mathbf{G}$-unitary connection on $V_- \otimes \mathbb{C}$. By the abstract definition of the Dolbeault operator on $\mathcal{Q}_\ell$, we have that the $(0,1)$-part of the right hand side of \eqref{eq:Nablaminusbisabs} coincides with $\overline{\partial}^\ell$. Formula \eqref{eq:NablaChernexp} follows from \cite[Equation (5.10)]{GaRuTi1}.
\end{proof}
\end{proposition}

In our next result we calculate an explicit formula for the generalized Hermitian metric $\mathbf{G}$ in terms of the isomorphism $\mathcal{Q}_\ell \cong \mathcal{Q}_{P,2i\partial \omega,\theta}$ in Lemma \ref{l:liftingQ}.

\begin{lemma}\label{t:Ggeneralized1} 
Let $X$ be a complex manifold endowed with a smooth Courant algebroid of string type $E$ and a compatible generalized metric $E = V_+ \oplus V_-$. Then, the Hermitian isometry $\psi \colon \mathcal{Q}_{P,2i\partial \omega,\theta} \to V_-\otimes \mathbb{C}$ induced by Lemma \ref{l:liftingQ} is given by
$$
\psi(V + r + \xi) = e^{i\omega}V + r - \tfrac{1}{2} e^{-i\omega} g^{-1}\xi.
$$
Consequently,
$$
\psi^*\mathbf{G}(V + r + \xi,V + r + \xi) = g(V,\overline{V}) + \tfrac{1}{4}g^{-1}(\xi,\overline{\xi}) -  \IP{r,\overline{r}} .
$$
\begin{proof}
The formula for $\psi$ follows by composing the isomorphisms
$$
\mathcal{Q}_{P,2i\partial\omega,\theta}\overset{e^{i\omega}}{\longrightarrow}\mathcal{Q}_{\ell}\overset{\pi_-}{\longrightarrow}V_-\otimes\mathbb{C}.
$$
The pullback of $\mathbf{G}$ in Definition (\ref{d:generalizedmetriHerm}) along $\psi$ is straightforward.
\end{proof}
\end{lemma}

\begin{remark}\label{rem:signature}
By the previous lemma, the signature of $\mathbf{G}$ is $(4n + 2l_2,2l_1)$, where $(l_1,l_2)$ is the signature of $\IP{,}: \mathfrak{k} \otimes \mathfrak{k} \to \mathbb{R}$ and $n = \dim_\CC X$.
\end{remark}

\subsection{Curvature of generalized Hermitian metrics}\label{sec:curvature}

We calculate next the curvature form and second Ricci curvature (see Equation \eqref{eq:secondRicciG}) of the generalized Hermitian metric $\mathbf{G}$ in Definition \ref{d:generalizedmetriHerm}. % We will continue with the notation in Section \ref{sec:GHermitian}. In particular, 
We will systematically use the identifications $\mathcal{Q}_\ell \cong V_- \otimes \mathbb{C}$ and 
\begin{equation}\label{eq:V-iso}
V_- \cong T \oplus \ad P_K.
\end{equation}
Consider the (possibly) indefinite metric on $V_-$ given by
$$
\IP{V+r,V+r}^0 := - \IP{V - g(V) +r,V - g(V) +r} = g(V,V) - \IP{r,r}.
$$
Then, extending $\CC$-linearly $\IP{,}^0$ to $V_- \otimes \CC$, it follows from Definition \ref{d:generalizedmetriHerm} that $\mathbf{G}$ is given by
$$
\mathbf{G}(s_1,s_2) = \IP{s_1,\overline{s_2}}^0.
$$
By Proposition \ref{t:Cherngeneralized}, the Chern connection $D^{\mathbf{G}}$ is the $\mathbb{C}$-linear extension of a $\IP{,}^0$-orthogonal connection
$$
D \colon \Omega^0(V_-) \to \Omega^1(V_-),
$$
and hence to calculate $F_\mathbf{G}:= F_{D^\mathbf{G}}$ it suffices to give a formula for $F_D$. Explicitly, in terms of the decomposition \eqref{eq:V-iso} we have
$$
D_V(W + r) =  \nabla^-_V W - g^{-1} \IP{i_V F_\theta,r} + d^\theta_V r - F_\theta(V,W).
$$
For the calculations, it will be useful to express $D$ in matrix notation as
$$
D=\left(\begin{array}{cc}
    \nabla^{-} & \mathbb{F}^{\dagger}\\
    -\mathbb{F} & d^{\theta}
    \end{array}\right)
$$
where $\mathbb{F} \in \Omega^1(\Hom(T,\ad P_K))$ is the $\Hom(T,\ad P_K)$-valued $1$-form
\begin{equation}\label{eq:Foperator}
(i_V\mathbb{F})(W) := F_\theta(V,W)
\end{equation}
and $\mathbb{F}^{\dagger} \in \Omega^1(\Hom(\ad P_K,T))$ is the corresponding $\IP{,}^0$-adjoint
$$
(i_V \mathbb{F}^{\dagger})(r) = - g^{-1} \IP{i_V F_\theta,r}.
$$
We will use the standard notation $R_{\nabla^-}$ for the curvature of $\nabla^-$ and also $\nabla^{\theta,-}$ for the covariant derivative induced by $\theta$ and $\nabla^-$ on $\Lambda^2T^* \otimes \ad P_K$. In particular,
$$
(\nabla^{\theta,-}_Z F_\theta)(V,W) = d_Z^\theta(F_\theta(V,W))- F_\theta(\nabla^-_ZV,W) - F_\theta(V,\nabla^-_ZW)
$$
for any triple of vector fields $V,W,Z$ on $M$.

\begin{lemma}\label{lemma:connectioncurvature}
The curvature of $D$ is given by
$$
F_{D} = \left(\begin{array}{cc}
    R_{\nabla^{-}}-\mathbb{F}^{\dagger}\wedge \mathbb{F}& - \mathbb{I}^{\dagger}\\
    \mathbb{I} & [F_\theta,]-\mathbb{F}\wedge \mathbb{F}^{\dagger}
\end{array}\right) 
$$
where
\begin{align*}
i_Wi_V\mathbb{F}^{\dagger}\wedge \mathbb{F}(Z) & = g^{-1}\IP{i_W F_\theta,F_\theta(V,Z)} - g^{-1}\IP{i_V F_\theta,F_\theta(W,Z)},\\
i_Wi_V \mathbb{I}(Z) & = (\nabla^{\theta,-}_ZF_\theta)(V,W) - F_\theta(V,g^{-1}i_Zi_Wd^c\omega) + F_\theta(W,g^{-1}i_Zi_Vd^c\omega),\\
i_Wi_V\mathbb{F}\wedge \mathbb{F}^{\dagger}(r) & = F_\theta(W,g^{-1}\IP{i_VF_\theta,r}) - F_\theta(V,g^{-1}\IP{i_WF_\theta,r}).
\end{align*}
\end{lemma}

\begin{proof}
To compute the curvature, we write
$$
D= D^0 + \left(\begin{array}{cc}
    0 &  \mathbb{F}^{\dagger}\\
    -\mathbb{F} & 0\\
\end{array}\right)
$$
where $D^0 = \nabla^- \oplus d^\theta$. Then, we have
\begin{equation}\label{eq:FDabstract}
\begin{split}
F_{D} & =F_{D_{0}}+d^{D_{0}}\left(\begin{array}{cc}
    0 &  \mathbb{F}^{\dagger}\\
    -\mathbb{F} & 0\\
\end{array}\right)+\left(\begin{array}{cc}
    -\mathbb{F}^{\dagger}\wedge \mathbb{F} & 0\\
    0 & -\mathbb{F}\wedge \mathbb{F}^{\dagger}\\
\end{array}\right) \\
& =\left(\begin{array}{c c}
    R_{\nabla^{-}}-\mathbb{F}^{\dagger}\wedge \mathbb{F} & (d^{\theta,-}\mathbb{F})^{\dagger}\\
    -d^{\theta,-}\mathbb{F} & [F_{\theta},]-\mathbb{F}\wedge \mathbb{F}^{\dagger}
\end{array}\right)
\end{split}
\end{equation}
where $d^{\theta,-} \colon \Omega^1(\Hom(T,\ad P_K)) \to \Omega^2(\Hom(T,\ad P_K))$ is the exterior covariant derivative induced by $\nabla^-$ and $\theta$. The explicit formulae for $\mathbb{F}^{\dagger}\wedge \mathbb{F}$ and $\mathbb{F}\wedge \mathbb{F}^{\dagger}$ above are algebraic and are left to the reader. As for $d^{\theta,-}\mathbb{F}$, we have
\begin{align*}
- i_Wi_V d^{\theta,-}\mathbb{F}(Z) & = - d^\theta_V(F_\theta(W,Z)) + d^\theta_W(F_\theta(V,Z)) + F_\theta([V,W],Z)\\
& + F_\theta(W,\nabla^-_V Z) -  F_\theta(V,\nabla^-_W Z)\\
& = d^\theta_Z(F_\theta(V,W)) + F_\theta([V,Z],W) - F_\theta([W,Z],V)\\
& + F_\theta(W,\nabla^-_V Z) -  F_\theta(V,\nabla^-_W Z)\\
& =  (\nabla^{\theta,-}_ZF_\theta)(V,W) + F_\theta(W,T_{\nabla^-}(V,Z)) -  F_\theta(V,T_{\nabla^-}(W,Z))
\end{align*}
where $T_{\nabla^-}$ denotes the torsion tensor of $\nabla^-$ and in the second equality we have used the Bianchi identity $d^\theta F_\theta = 0$. Our formula for $\mathbb{I}$ follows now from $T_{\nabla^-}(W,Z) = g^{-1}i_Zi_Wd^c\omega$.
\end{proof}

We next calculate the \emph{second Ricci curvature} of the generalized Hermitian metric $\mathbf{G}$, defined by the expression
\begin{equation}\label{eq:secondRicciG}
S_{\mathbf{G}} \frac{\omega^n}{n} = F_\mathbf{G} \wedge \omega^{n-1}
\end{equation}
where $\omega$ is the Hermitian form in Lemma \ref{l:lifting2}. Similarly as before, the skew-Hermitian endomorphism $S_{\mathbf{G}}$ is given by the $\mathbb{C}$-linear extension of the second Ricci curvature $S_{D}$ of the connection $D$. To calculate $S_D$ below, we need the following technical lemma.

\begin{lemma}\label{lem:Hodgestar}
Let $(M,g)$ be a Riemannian manifold of even dimensions. Let $F \in \Omega^2$ and $H \in \Omega^3$ be differential forms. Then, the Hodge star operator satisfies:
$$
 i_{V}*(F \wedge * H) = \frac{1}{2}\sum_{i=1}^m F(e_i,g^{-1}i_V i_{e_i}H)
$$
for any vector field $V$ and any choice of $g$-orthonormal frame  $e_1, \ldots, e_{m}$ of $T$.
\end{lemma}

\begin{proof}
For $e^i$ the dual frame, one has 
$$
*(e^i\wedge *\psi)=i_{e_i}\psi
$$
for $\psi\in \Omega^p$, and therefore
\begin{align*}
    *(e^i\wedge e^j\wedge * \psi ) % & = (-1)^{p-1}*(e^i\wedge *(*(e^j\wedge * \psi)))\\
    % & = (-1)^{p-1}*(e^i\wedge *(i_{e_j}\psi))\\
    %& =(-1)^{p-1} i_{e_i}i_{e_j}\psi\\
    & = (-1)^p i_{e_j}i_{e_i}\psi.
\end{align*}
By bilinearity, we get
$$
*(F \wedge *H)=-\sum_{i<j}F (e_i,e_j)H(e_i,e_j,\cdot),
$$
and therefore
\begin{align*}
    i_{V}*(F\wedge * H)  = - \sum_{i<j}F(e_i,e_j)H(e_i,e_j,V) 
   % & = -\frac{1}{2}\sum_{i,j}F_\theta(e_i,e_j)H(e_i,e_j,V)\\
    % & = \frac{1}{2}\sum_i g(F_\theta(e_i,\cdot),H(e_i,V,\cdot))\\
    = \frac{1}{2}\sum_i F(e_i,g^{-1}i_V i_{e_i}H).
\end{align*}

\end{proof}

Recall that the \emph{Bismut connection} of the Hermitian metric $g$ in Lemma \ref{l:lifting2} is given by (cf. Proposition \ref{t:Cherngeneralized})
$$
\nabla^B = \nabla - \tfrac{1}{2}g^{-1}d^c\omega.
$$
for $\nabla$ the Levi-Civita connection of $g$. This expression defines a unitary connection on the tangent bundle of $X = (M,J)$, and hence it induces a well-defined curvature on the anti-canonical bundle $- i \rho_B$, where $\rho_B$ is the \emph{Bismut Ricci form}. Explicitly, for a choice of $g$-orthonormal basis $e_1, \ldots, e_{2n}$ of $T$ at a point, one has
\begin{equation}\label{eq:secondRicci}
\rho_B(V,W) = \frac{1}{2} \sum_{j=1}^{2n} g(R_{\nabla^B}(V,W)Je_j,e_j).
\end{equation}

\begin{proposition}\label{prop:secondRicci}
The second Ricci form $S_{D}$ of the connection $D$ is given by
$$
S_{D} = \left(\begin{array}{cc}
    - g^{-1}(\rho_B + \IP{S_\theta,F_\theta}) & - \mathbb{S}^{\dagger}\\
    \mathbb{S} & [S_\theta,]
\end{array}\right) 
$$
where
\begin{align*}
\mathbb{S}(V) & =i_{JV}\Bigg{(}-d^{\theta *} F_\theta- i_{\theta_\omega^\sharp}F_\theta + *(F_\theta\wedge * d^c\omega)) \Bigg{)},
\end{align*}
%\begin{align*}
%\mathbb{S}(V) & =i_{JV}\Bigg{(}-d^{\theta *} F_\theta- i_{\theta_\omega^\sharp}F_\theta +\frac{1}{2}\sum_{j=1}^{2n} F_\theta(e_j,g^{-1}d^c\omega(e_j,\cdot)) \Bigg{)},
%\end{align*}
for $d^{\theta *}$ the adjoint of $d^\theta$ and $\theta_\omega=Jd^*\omega$ the Lee form of $g$.
\end{proposition}

\begin{proof}
Recall the alternative expression for the second Ricci form
$$
S_D = \frac{1}{2}\sum_{j=1}^{2n} F_D(e_j,J e_j)
$$
for a choice of $g$-orthonormal basis $e_1, \ldots, e_{2n}$ of $T$ at a point. Using this and applying Lemma \ref{lemma:connectioncurvature}, we have
$$
S_{D} = \left(\begin{array}{cc}
    S_{\nabla^{-}}-\tfrac{1}{2}\mathbb{F}^{\dagger}\wedge \mathbb{F}(e_i,Je_i)& - \tfrac{1}{2}\mathbb{I}^{\dagger}(e_i,Je_i)\\
    \tfrac{1}{2}\mathbb{I}(e_i,Je_i) & [S_\theta,] - \tfrac{1}{2}\mathbb{F}\wedge \mathbb{F}^{\dagger}(e_i,Je_i)
\end{array}\right).
$$
We first calculate
\begin{align*}
g(i_{Je_i}i_{e_i}\mathbb{F}^{\dagger}\wedge \mathbb{F}(V),W) & = \IP{F_\theta(Je_i,W),F_\theta(e_i,V)} - \IP{F_\theta(e_i,W),F_\theta(Je_i,V)}
\end{align*}
Combining this with the identity (see the proof of \cite[Proposition 3.21]{GaSt})
\begin{equation}\label{eq:BismutHulleq}
g(R_{\nabla^-}(V_1,V_2)V_3,V_4) - g(R_{\nabla^B}(V_3,V_4),V_1,V_2) = \tfrac{1}{2}dd^c \omega(V_1,V_2,V_3,V_4)
\end{equation}
and Lemma \ref{l:lifting2}, we also obtain
\begin{align*}
g(S_{\nabla^-}(V),W) & = \tfrac{1}{2} g(R_{\nabla^-}(e_i,Je_i)V,W)\\
& = \tfrac{1}{2} g(R_{\nabla^B}(V,W)e_i,Je_i) + \tfrac{1}{4}dd^c\omega(e_i,Je_i,V,W)\\
& = - \rho_B(V,W) - \tfrac{1}{4}\IP{F_\theta \wedge F_\theta}(e_i,Je_i,V,W)\\
& = - \rho_B(V,W) - \tfrac{1}{2}\IP{i_{e_i}F_\theta \wedge F_\theta}(Je_i,V,W)\\
& = - \rho_B(V,W) - \IP{S_\theta , F_\theta(V,W)} + \tfrac{1}{2}\IP{i_{e_i}F_\theta \wedge i_{J e_i}F_\theta}(V,W)\\
& = - \rho_B(V,W) - \IP{S_\theta , F_\theta(V,W)} + \tfrac{1}{2}\IP{F_\theta(e_i,V),F_\theta(Je_i,W)}\\
&  - \tfrac{1}{2}\IP{F_\theta(e_i,W),F_\theta(Je_i,V)}\\
& = - \rho_B(V,W) - \IP{S_\theta , F_\theta(V,W)} + \tfrac{1}{2}g(i_{Je_i}i_{e_i}\mathbb{F}^{\dagger}\wedge \mathbb{F}(V),W),
\end{align*}
as claimed. Using again Lemma \ref{l:lifting2}, in particular $F_\theta = F_\theta^{1,1}$, we also have 
\begin{align*}
i_{Je_i}i_{e_i}\mathbb{F}\wedge \mathbb{F}^{\dagger}(r) & = F_\theta(Je_i,g^{-1}\IP{i_{e_i}F_\theta,r}) - F_\theta(e_i,g^{-1}\IP{i_{Je_i}F_\theta,r})\\
& = - F_\theta(e_i,J g^{-1}\IP{i_{e_i}F_\theta,r}) - F_\theta(e_i,g^{-1}\IP{F_\theta(Je_i,),r})\\
& =-2F_\theta(e_i,g^{-1}\langle F_\theta(Je_i,),r\rangle)\\
& = -2F_\theta(e_i,e_j)\langle F_\theta(Je_i,e_j),r\rangle.
\end{align*}
Finally, the last expression vanishes using again $F_\theta=F_\theta^{1,1}$ and symmetry considerations.

In the computation of the remaining term, we will use the following standard expressions for the covariant derivative of the almost complex structure $J$, the adjoint of $d^\theta$, and the Lee form: %and the definition of $\nabla^-$:
\begin{align*}
    (\nabla_V J) W & = \frac{1}{2}g^{-1}(d\omega(V,W,\cdot)-d^c\omega(JV,W,\cdot)),\\
%    (\nabla^-_XY - \nabla_X Y & = \frac{1}{2}g^{-1}(d^c\omega(X,Y,\cdot)).
d^{\theta *} F_\theta & = - i_{e_i} \nabla^{\theta,g}_{e_i} F_\theta,\\ 
\theta_\omega (V) & = \tfrac{1}{2} d\omega(e_i,Je_i,V).
\end{align*}
where $\nabla^{\theta,g}$ is the covariant derivative with respect to the Levi-Civita connection $\nabla$ and $\theta$. Combining this with \eqref{eq:FDabstract}, we conclude that:
\begin{align*}
i_{Je_i}i_{e_i}\mathbb{I}(V) & = -i_{Je_i}i_{e_i}d^{\theta,-}\mathbb{F}(V)\\
& = i_{Je_i}i_{e_i}d^{\theta}(i_VF_\theta) - F_\theta(e_i,\nabla^-_{Je_i}V)+F_\theta(Je_i,\nabla^-_{e_i}V)\\
& = 2d^\theta_{e_i}(F_\theta(e_i,JV))+F_\theta([e_i,Je_i],V)+2F_\theta(Je_i,\nabla^-_{e_i}V)\\
& =2(\nabla^{\theta,g}_{e_i}F_\theta)(e_i,JV)+2F_\theta(\nabla_{e_i}e_i,JV)+2F_\theta(e_i,\nabla_{e_i}JV)\\
&+2F_\theta(\nabla_{e_i}Je_i,V)+2F_\theta(Je_i,(\nabla^--\nabla)_{e_i} V)+2F_\theta(Je_i,\nabla_{e_i}V)\\
& =- 2 d^{\theta *}F_\theta(JV)+2F_\theta((\nabla_{e_i}J)e_i,V)+2F_\theta(Je_i,(\nabla^--\nabla)_{e_i}V)\\
& +2F_\theta(e_i,(\nabla_{e_i}J)V)\\
& =-2d^{\theta *}F_\theta(JV)-2F_\theta(\theta_\omega^\sharp,JV)+F_\theta(e_i,g^{-1}i_{JV}i_{e_i}d^c\omega).
\end{align*}
The statement follows from Lemma \ref{lem:Hodgestar}.
\end{proof}

\section{Coupled Hermitian-Einstein metrics}\label{sec:HE}

\subsection{The coupled Hermitian-Einstein system}\label{sec:BHE}

We introduce next a relaxed version of the Hull-Strominger system \eqref{eq:HSintro}, motivated by Proposition \ref{prop:secondRicci}. Our definition will provide a natural class of generalized Hermitian metrics $\mathbf{G}$ satisfying the following Hermitian-Einstein type equation
\begin{equation}\label{eq:HE}
F_\mathbf{G} \wedge \omega^{n-1} = 0.
\end{equation}
Observe that, in the previous expression, the classical metric $g$ appears both in the definition of the generalized Hermitian metric $\mathbf{G}$ and in the trace operator $ \wedge \omega^{n-1}$. Motivated by this additional non-linearity in \eqref{eq:HE}, we will call these new equations the \emph{coupled Hermitian-Einstein system} (cf. \cite{GaJoSt}).

We follow the setup in Section \ref{sec:BC}, which we recall briefly. Let $G$ be a complex reductive Lie group with Lie algebra $\mathfrak{g}$. Fix a non-degenerate bi-invariant symmetric bilinear form $\IP{,}: \mathfrak{g} \otimes \mathfrak{g} \to \mathbb{C}$ such that $\IP{\mathfrak{k},\mathfrak{k}} \subset \RR$ for any compact Lie subalgebra $\mathfrak{k} \subset \mathfrak{g}$.

\begin{definition}\label{def:BHEsystem}
Let $X$ be a complex manifold of complex dimension $n$ endowed with a holomorphic principal $G$-bundle $P$. We say that a pair $(g,h)$, where $g$ is a Hermitian metric on $X$ and $h$ is a reduction of $P$ to a maximal compact subgroup, satisfies the \emph{coupled Hermitian-Einstein system} if
\begin{equation}\label{eq:BHE}
\begin{split}
F_h\wedge \omega^{n-1} & = \frac{z}{n} \omega^n,\\
\rho_{B} + \langle z , F_h\rangle & = 0,\\
dd^{c}\omega+\langle F_h\wedge F_h\rangle & = 0,\\
\end{split}
\end{equation}
where %$\omega = g(J,)$, 
$F_h$ is the curvature of the Chern connection of $h$, $\rho_{B}$ is the Bismut-Ricci form of the metric $g$, and $z$ is a central element in the Lie algebra $\mathfrak{k}$.
\end{definition}

\begin{remark}
Assuming that $X$ is compact, the system \eqref{eq:BHE} imposes some obvious necessary constraints, namely
$$
2\pi c_1(X) + [\langle z , F_h\rangle] = 0 \in H^2_{dR}(X,\RR),\qquad p_1(P) = 0 \in H^{2,2}_{BC}(X,\RR),
$$
where $[\langle z , F_h\rangle]$ is the Chern-Weyl invariant associated to $P$ via the character $\langle z, \rangle \colon \mathfrak{k} \to \mathbb{R}$ and $p_1$ is defined as in \eqref{eq:p1(P)=0dRC}. Furthermore, by the extension of the Donaldson-Uhlenbeck-Yau Theorem to Hermitian manifolds \cite{Buchdahl,LiYauHYM,lt}, the bundle $P$ must be slope polystable with respect to the unique Gauduchon metric in the conformal class of $g$ (cf. \cite{AnBiwas}). %[\textbf{MGF: la razón por la que esto es cierto es porque la poliestabilidad para fibrados principales es equivalente a la del adjunto.}]
\end{remark}

This section is devoted to the study of basic structural properties of the coupled Hermitian-Einstein system \eqref{eq:BHE}. Our first goal is to prove that any solution induces a generalized Hermitian metric $\mathbf{G}$ on a Bott-Chern algebroid which satisfies the Hermitian-Einstein equation \eqref{eq:HE}. %Furthermore, solutions of \eqref{eq:BHE} provide a natural class of generalized Ricci flat metrics on string algebroids, for a suitable choice of divergence operator \cite{Ga1}.

The following lemma can be compared with the classical result which states that a Hermitian-Yang-Mills connection is Yang-Mills, provided that the background metric is K\"ahler. The analogue in Hermitian geometry is apparently well-known to experts but, since we have not been able to find it in the literature, we shall provide a complete proof here.

\begin{lemma}\label{lemma:pHYM1}
Let $(X,g)$ be a Hermitian manifold of complex dimension $n$ endowed with a holomorphic principal $G$-bundle $P$. Let $h$ be a Hermitian-Einstein reduction on $P$, that is, satisfying
$$
F_h \wedge \omega^{n-1} = \frac{z}{n} \omega^n
$$
for $z$ a central element in the Lie algebra $\mathfrak{k}$. Then, the following equation is satisfied
$$
d^{h*} F_h + i_{\theta_\omega^\sharp}F_h - *(F_\theta\wedge * d^c\omega) = 0.
$$
%$$
%d^*_h F_h + i_{\theta_\omega^\sharp}F_h- \frac{1}{2}\sum_{j=1} ^{2n}F_h(e_j,g^{-1}d^c\omega(e_j,\cdot)) = 0.
%$$
\end{lemma}

\begin{proof}
The curvature form $F_h$ satisfies
$$
F_h=F_h^{1,1} \qquad d^h F_h=0
$$
and the Hermitian-Einstein condition can be rewritten as
$$
\frac{1}{2}\sum_i F_h(e_i,Je_i) = z.
$$
Using the conditions above, we obtain
\begin{align*}
    d^{h*} F_h(V) & = -(\nabla^{h,g}_{e_i}F_h)(e_i,V)\\
    & = -d^h_{e_i}(F_h(Je_i,JV))+F_h(\nabla_{e_i}e_i,V)+F_h(e_i,\nabla_{e_i}V)\\
    & = - d^hF_h(e_i,Je_i,JV) - d^h_{Je_i}(F_h(e_i,JV)) + d^h_{JV}(F_h(e_i,Je_i))\\
    & -F_h([e_i,Je_i],JV)+F_h([e_i,JV],Je_i) -F_h([Je_i,JV],e_i)\\
    & +F_h(\nabla_{e_i}e_i,V)+F_h(e_i,\nabla_{e_i}V)\\
    & = (\nabla^{h,g}_{Je_i}F_h)(Je_i,V)) + F_h(\nabla_{Je_i} Je_i,V) + F_h (Je_i,\nabla_{Je_i}V)\\
    & -F_h([e_i,Je_i],JV)+F_h([e_i,JV],Je_i) -F_h([Je_i,JV],e_i)\\
    & +F_h(\nabla_{e_i}e_i,V)+F_h(e_i,\nabla_{e_i}V)\\
    & = -d^{h *} F_h(V)+2F_h(\nabla_{e_i}e_i,V)+2F_h(e_i,\nabla_{e_i}V)\\
    & -2F_h(\nabla_{e_i}Je_i,JV) +2F_h(\nabla_{e_i}JV,Je_i)+2F_h(\nabla_{JV}Je_i,e_i) .
\end{align*}
Collecting the terms $d^{h *} F_h(V)$  and using again that $F_h = F_h^{1,1}$, we have
\begin{align*}
d^{h*} F_h(V)  & = -F_h((\nabla_{e_i}J)e_i,JV)+F_h((\nabla_{e_i}J)V,Je_i)+F_h(\nabla_{JV}Je_i,e_i).
\end{align*}
Using elementary symmetry properties, which imply
$$
F_h(\nabla_{JV}Je_i,e_i) = F_h(e_j,e_i)g(\nabla_{JV}Je_i,e_j)  = 0,
$$
combined with the formulae for $\nabla J$ and $\theta_\omega$ in the proof of Proposition \ref{prop:secondRicci}, which imply
\begin{align*}
F_h((\nabla_{e_i}J)e_i,JV)  & = F_h(\theta_\omega^\sharp,V),\\
F_h((\nabla_{e_i}J)V,J e_i)  & = \tfrac{1}{2} F_h(e_i,g^{-1}i_V i_{e_i} d^c\omega),
\end{align*}
the claim now follows from Lemma \ref{lem:Hodgestar}.
\end{proof}

In the next result we characterize the Hermitian-Einstein equation \eqref{eq:HE} for a generalized Hermitian metric.  

\begin{lemma}\label{lem:HE}
Let $X$ be a complex manifold endowed with a holomorphic principal $G$-bundle $P$. Assume that a pair $(g,h)$ satisfies the Bianchi identity
\begin{equation}\label{eq:Bianchi}
dd^c \omega + \IP{F_h \wedge F_h} = 0.
\end{equation}
Consider the holomorphic vector bundle underlying the Bott-Chern algebroid $\mathcal{Q}_{P,2i\partial\omega,\theta^h}$ (see Example \ref{def:Q0}) endowed with the (possibly) indefinite Hermitian metric $\mathbf{G}$ in Lemma \ref{t:Ggeneralized1}. Then, $\mathbf{G}$ solves the Hermitian-Einstein equation \eqref{eq:HE} if and only if the following conditions hold
\begin{equation}\label{eq:HEexplicit}
\begin{split}
[S_h, ] & =0,\\
d^{h*} F_h + i_{\theta_\omega^\sharp}F_h - *(F_\theta\wedge * d^c\omega) & = 0,\\
\rho_B + \IP{S_h,F_h} & = 0,
\end{split}
\end{equation}
where $S_h$ denotes the second Ricci curvature of $h$.
%$$
%-d^{\theta *} F_\theta- i_{\theta_\omega^\sharp}F_\theta +\tfrac{1}{2}\sum_i F_\theta(e_i,g^{-1}d^c\omega(e_i,\cdot)) & =0
%$$

\end{lemma}

\begin{proof}
By construction, the Chern connection of $\mathbf{G}$ is the $\mathbb{C}$-linear extension of the $\IP{,}^0$-orthogonal connection $D$ in Section \ref{sec:curvature}. Then, the proof is straightforward from Proposition \ref{prop:secondRicci}.
\end{proof}

%\begin{remark}
%When $G = \{1\}$, Lemma \ref{lem:HE} recovers \cite[Proposition 4.4]{GaJoSt} for pluriclosed Bismut-Hermitian-Einstein metrics.
%\end{remark}

%\begin{remark}
As one can directly see from \eqref{eq:HEexplicit}, the Hermitian-Einstein condition \eqref{eq:HE} for a generalized Hermitian metric is very sensitive to the choice of quadratic Lie algebra $(\mathfrak{g},\IP{,})$. For example, when $\mathfrak{g}$ is abelian the first condition is trivially satisfied. In particular, when $\mathfrak{g} = \{0\}$, Lemma \ref{lem:HE} recovers \cite[Proposition 4.4]{GaJoSt} for pluriclosed Hermitian metrics. On the other extreme, when $\mathfrak{g}$ is semisimple, the first equation implies that $S_h = 0$ and hence the second equation is satisfied by Lemma \ref{lemma:pHYM1}. Furthermore, in this case one has $\rho_B = 0$ and hence the equations \eqref{eq:HEexplicit} combined with the Bianchi identity \eqref{eq:Bianchi} are equivalent to the coupled Hermitian-Einstein system \eqref{eq:BHE} with $z = 0$.
%\end{remark}

We are ready to prove the first main result of this section.

\begin{proposition}\label{prop:BHE-HE}
Let $X$ be a complex manifold endowed with a holomorphic principal $G$-bundle $P$. Assume that $(g,h)$ solves the coupled Hermitian-Einstein system \eqref{eq:BHE}. Consider the holomorphic vector bundle $\mathcal{Q}_{P,2i\partial\omega,\theta^h}$ endowed with the (possibly) indefinite Hermitian metric $\mathbf{G}$ in Lemma \ref{t:Ggeneralized1}. Then, $\mathbf{G}$ solves the Hermitian-Einstein equation \eqref{eq:HE}.
\end{proposition}

\begin{proof}
The proof is straightforward, combining Definition \ref{def:BHEsystem} with Lemma \ref{lemma:pHYM1} and Lemma \ref{lem:HE}.
\end{proof}

To finish this section we provide a Riemannian characterization of the coupled Hermitian-Einstein system. In particular, we will see that the solutions of \eqref{eq:BHE} correspond to a natural class of generalized Ricci flat metrics on string algebroids and exhibit an interesting relation to heterotic supergravity, giving further motivation for their study (see Remark \ref{rem:GRF} and Remark \ref{rem:heterotic}).

\begin{proposition}\label{prop:BHE-GRF}
Let $X$ be a complex manifold endowed with a holomorphic principal $G$-bundle $P$. Assume that $(g,h)$ solves the coupled Hermitian-Einstein system \eqref{eq:BHE}. Then, $(g,h)$ solves the equations
\begin{equation}\label{eq:GRF}
\begin{split}
\operatorname{Rc} - \frac{1}{4}H^2 + \IP{i_{e_i}F_\theta , i_{e_i}F_\theta} + \tfrac{1}{2}L_{\varphi^\sharp} g & = 0, \\
d^{*} H - d \varphi +  i_{\varphi^\sharp}H & =0,\\
d^{\theta *} F_\theta +  i_{\varphi^\sharp}F_\theta + *(F_\theta\wedge * H) & =0,
\end{split}
\end{equation}
where $\operatorname{Rc}$ is the Riemannian Ricci tensor and
\begin{equation}\label{eq:Hvarphi}
H = -d^c \omega, \qquad \theta = \theta^h, \qquad \varphi = \theta_\omega.
\end{equation}
%Conversely, a solution of \eqref{eq:GRF} with $g$ Hermitian and satisfying \eqref{eq:Hvarphi} solves the Bismut-Hermitian-Einstein system if \eqref{eq:Bianchi} and $S_h = 0$ hold.
\end{proposition}

\begin{proof}
We have already seen that the Hermitian-Einstein condition implies the last equation in \eqref{eq:GRF}. Therefore, it is enough to prove that \eqref{eq:BHE} implies that
\begin{equation}\label{eq:rhoBdecomp}
\begin{split}
    \rho_B^{1,1}(\cdot,J\cdot) & = \operatorname{Rc} - \frac{1}{4}H^2 + \IP{i_{e_i}F_\theta , i_{e_i}F_\theta}+\IP{z,F_\theta(J,)} + \tfrac{1}{2}L_{\varphi^\sharp} g,\\
    \rho_B^{2,0+0,2}(\cdot,J\cdot) & = -\frac{1}{2}( d^{*} H - d \theta_\omega +  i_{\varphi^\sharp}H).
\end{split}
\end{equation}
To check this we will use the following general formulae, valid on any Hermitian manifold (see \cite[Proposition 3.1]{IvPa}):
\begin{align*}
    \operatorname{Rc}(V,W) & =  \operatorname{Rc}_B(V,W)-\tfrac{1}{2}d^*d^c\omega(V,W)+\tfrac{1}{4}g(d^c\omega(V,e_i),d^c\omega(W,e_i)),\\
    \rho_B(V,W) & = -\operatorname{Rc}_B(V,JW)-(\nabla^B_V \theta_\omega)JW+\tfrac{1}{4}dd^c\omega(V,W,e_i,Je_i),
\end{align*}
where $\operatorname{Rc}_B$ denotes the Ricci tensor of $\nabla^B$. To prove the first identity in \eqref{eq:rhoBdecomp}, we now calculate
\begin{align*}
    \rho_B^{1,1}(V,JW) &  = \tfrac{1}{2}\left(\rho_B(V,JW)-\rho_B(JV,W)\right)\\
    & = \tfrac{1}{2}(\operatorname{Rc}_B(V,W) + \operatorname{Rc}_B(JV,JW) + (\nabla^B_V \theta_\omega)W+(\nabla^B_{JV}\theta_\omega)JW\\
    &+\tfrac{1}{4}dd^c\omega(V,JW,e_i,Je_i)-\tfrac{1}{4}dd^c\omega(JV,W,e_i,Je_i) )\\
    & = \tfrac{1}{2}(2 \operatorname{Rc}(V,W)-\tfrac{1}{2}g(d^c\omega(V,e_i,),d^c\omega(W,e_i,))\\
    & - \tfrac{1}{2}\langle F_\theta\wedge F_\theta\rangle(V,JW,e_i,Je_i)+L_{\theta_\omega^\sharp}g(V,W))\\
    & = \Big{(}\operatorname{Rc} - \tfrac{1}{4}H\circ H +\langle i_{e_i}F_\theta, i_{e_i}F_\theta\rangle+\langle z,F_{\theta}(J\cdot,\cdot)\rangle +\tfrac{1}{2}L_{\theta_\omega^\sharp}g\Big{)}(V,W)
\end{align*}
where we denote
$$
H\circ H=\sum_{i,j} H(e_i,e_j,\cdot)H(e_i,e_j,\cdot)
$$
and for the third equality we have used the identity (see \cite[Equation (3.23)]{IvPa})
$$
\operatorname{Rc}_B(W,JV) = -\operatorname{Rc}_B(V,JW) -(\nabla^B_V \theta_\omega)JW-(\nabla^B_W \theta_\omega)JV.
$$
Similarly, the second identity in \eqref{eq:rhoBdecomp} follows from
\begin{align*}
    \rho_{B}^{2,0+0,2}(V,JW)& =\tfrac{1}{2}(\rho_B(V,JW)+\rho_B(JV,W))\\
    & = \tfrac{1}{2}(\operatorname{Rc}_B(V,W)+(\nabla^B_V \theta_\omega)W+\tfrac{1}{4}dd^c\omega(V,JW,e_i,Je_i)\\
    & -\operatorname{Rc}_B(JV,JW)-(\nabla^B_{JV} \theta_\omega)JW+\tfrac{1}{4}dd^c\omega(JV,W,e_i,Je_i))\\
    &=\tfrac{1}{2}(\operatorname{Rc}_B(V,W) - \operatorname{Rc}_B(W,V)+(\nabla^B_V \theta_\omega)W-(\nabla^B_{W}\theta_\omega)V)\\
    & = \tfrac{1}{2}(d^*d^c\omega(V,W)+d\theta_\omega(V,W)+\theta_\omega(g^{-1}d^c\omega(V,W)))\\
    & =- \tfrac{1}{2}(d^*H-d\theta_\omega+i_{\theta_\omega^\sharp}H)(V,W).
\end{align*}
\end{proof}

\begin{remark}\label{rem:GRF}
Consider the smooth string algebroid in Example \ref{def:E0}. Applying \cite[Lemma 7.1]{Ga1}, equations \eqref{eq:GRF} correspond to the vanishing of the \emph{generalized Ricci tensor} $\operatorname{Ric}^+$ of the generalized metric $V_+$ (see \eqref{eq:Vpm}), for a suitable choice of divergence operator determined by $\varphi$. Thus, by the previous result, any solution to the coupled Hermitian-Einstein system \eqref{eq:BHE} is generalized Ricci flat. Note that \eqref{eq:GRF} corrects some mistakes in the formula for the generalized Ricci tensor in \cite{Ga0,Ga1}.%[\textbf{RGM: Hay m\'as diferencias: con respecto a \textit{Torsion-free...}:
%\begin{itemize}
%    \item Signo en el t\'ermino $F\circ F$
%    \item Signo en $*F\wedge * H$ y factor $\tfrac{1}{2}$ que proviene de un c\'alculo inmediatamente anterior (comparar con \ref{lem:Hodgestar} aqu\'i). 
%\end{itemize}
%Con respecto a \textit{Ricci flow, Killing spinors...}: los puntos anteriores + hay un signo relativo entre los t\'erminos en $\nabla^+\varphi$ (comp. 1 de $GRic^+$) y $F(g^{-1}\varphi,\cdot)$ (comp. 2)}]

\end{remark}

\begin{remark}\label{rem:heterotic}
When $\varphi = df$ for a smooth function $f$, the equations \eqref{eq:GRF} match the so called equations of motion of the heterotic supergravity in the mathematical physics literature (see e.g. \cite{FeIvUgVi}).
\end{remark}

\vspace{3mm}

\subsection{Relation to the Hull-Strominger system}\label{sec:HS-HE}

We study next the relation between the coupled Hermitian-Einstein system \eqref{eq:BHE} and the Hull-Strominger system \eqref{eq:HSintro}. Our construction requires the ansatz \eqref{eq:HYMintro} for the connection $\nabla$, and hence in our discussion we will implicitly assume this condition. We will embrace an abstract definition of the Hull-Strominger system, as considered in \cite[Definition 2.4]{GaRuShTi}, which is valid in arbitrary dimensions.

Let $X$ be a compact complex manifold of dimension $n$ endowed with a holomorphic volume form $\Omega$. Let $V_0$ and $V_1$ denote holomorphic vector bundles over $X$ satisfying
\begin{equation}\label{eq:c1c2abstract}
 ch_2(V_0) = ch_2(V_1) \in H^{2,2}_{BC}(X,\RR).
\end{equation}

\begin{definition}\label{def:HSabstract}
We say that a triple $(g,h_0,h_1)$, where $g$ is a Hermitian metric on $X$ and $h_j$ is a Hermitian metric on $V_j$, $j = 0,1$, satisfies the \emph{Hull-Strominger system} with coupling constant $\alpha \in \RR$ if
\begin{equation}\label{eq:HSabstract}
\begin{split}
F_{h_0}\wedge \omega^{n-1} & =0,\\
F_{h_1}\wedge \omega^{n-1} & =0,\\
d(\|\Omega\|_\omega \omega^{n-1}) & = 0,\\
dd^c \omega - \alpha \tr F_{h_0} \wedge F_{h_0} + \alpha \tr F_{h_1} \wedge F_{h_1} & = 0.
\end{split}
\end{equation}
\end{definition}

\begin{remark}\label{rem:HSabstract}
%[\textcolor{red}{Modify or delete. What we know: if $\varphi: V_0 \overset{\cong}{\rightarrow} T^{1,0}$ smooth iso. and have a solution to (\ref{eq:HSintro}) with $\nabla$ on $V_0$ Hermitian-Yang-Mills and $H_0$-unitary, then by gauge transformation by $\varphi$, have a solution to (\ref{eq:HSabstract}) for $(T^{1,0},\varphi_*\nabla^{0,1})$ and $h_0=\varphi_*H_0$, and conversely}]  
Take $n=3$, $V = V_1$ and assume that $V_0$ is isomorphic to $T^{1,0}$ as a smooth complex vector bundle. Then, any solution of \eqref{eq:HSintro} with the Hermitian-Yang-Mills ansatz \eqref{eq:HYMintro} determines a solution to \eqref{eq:HSabstract} with $V_0 = (T^{1,0},\nabla^{0,1})$ and $h_0 = g$. Conversely, any solution to \eqref{eq:HSabstract} gives a solution of \eqref{eq:HSintro} with $\nabla$ satisfying \eqref{eq:HYMintro}, defined by pulling-back the Chern connection $D^{h_0}$ via a complex gauge transformation on $T^{1,0}$ taking $g$ to $h_0$ (in order to obtain a connection which is $g$-compatible). Observe that the equations \eqref{eq:HSintro} are invariant under this change.
\end{remark}

In the next result we establish the relation with the coupled Hermitian-Einstein system.

\begin{proposition}\label{prop:HS-BHE}
Assume that $(X,\Omega,V_0,V_1)$ admits a solution $(g,h_0,h_1)$ of the Hull-Strominger system \eqref{eq:HSabstract} with coupling constant $\alpha$. Denote by $P$ the holomorphic principal $G$-bundle of split frames of $V_0 \oplus V_1$, $h = h_0 \oplus h_1$, and consider the pairing $\IP{,}: \mathfrak{g} \otimes \mathfrak{g} \to \mathbb{C}$ induced by 
\begin{equation}\label{eq:pairingHS}
\IP{,} : = - \alpha \tr_{V_0} + \alpha \tr_{V_1}.
\end{equation}
Then, $(g,h)$ solves the coupled Hermitian-Einstein system \eqref{eq:BHE} with $z = 0$.
\end{proposition}

\begin{proof}
By construction, it suffices to prove that the $g$ satisfies $\rho_B = 0$. This follows from the fact that $\rho_B$ is proportional to the curvature of the connection induced by $\nabla^B$ in the anti-canonical bundle (see \eqref{eq:secondRicci}) combined with the second equation in \eqref{eq:HSabstract}, which implies (see \cite[Proposition 3.6]{GFReview})
$$
\nabla^B(\|\Omega\|_\omega^{-1} \Omega) = 0.
$$
\end{proof}

The key upshot of the previous result is that any solution of the Hull-Strominger system \eqref{eq:HSabstract} determines uniquely a possibly indefinite Hermitian-Einstein generalized metric. The proof follows from Proposition \ref{prop:BHE-HE} and Proposition \ref{prop:HS-BHE}. 

\begin{proposition}\label{prop:HS-HE}
Assume that $(X,\Omega,V_0,V_1)$ admits a solution $(g,h_0,h_1)$ of the Hull-Strominger system \eqref{eq:HSabstract}. Consider the holomorphic vector bundle underlying the Bott-Chern algebroid $Q_{P,2i\partial\omega,\theta^h}$, where $P$ and $h$ are as in Proposition \ref{prop:HS-BHE}. Then, the (possibly) indefinite Hermitian metric $\mathbf{G}$ in Lemma \ref{t:Ggeneralized1} solves the Hermitian-Einstein equation \eqref{eq:HE}.
\end{proposition}

\begin{remark}\label{rem:DelaOssa}
Lemma \ref{lem:HE} and Proposition \ref{prop:HS-HE} shall be compared with the original result by De la Ossa, Larfors, and Svanes in \cite[Corollary 1]{OssaLarforsSvanes}, who observed that the Hull-Strominger system is equivalent to \eqref{eq:HE} to all orders in perturbation theory.
\end{remark}

\section{Futaki invariants}\label{sec:Futaki}

\subsection{Aeppli classes and Futaki Invariants}\label{sec:Futakidef}

In this section we introduce a family of characters which obstructs the existence of solutions to the Hull-Strominger system \eqref{eq:HSabstract}. This stems from a general formalism that associates an invariant to any equation with a moment map interpretation (see e.g. \cite{AGG}), which we call Futaki invariant by analogy with the classical invariant obstructing the existence of K\"ahler-Einstein metrics on a K\"ahler manifold \cite{Futaki1}. 

The construction of our Futaki invariants follows from the application of this general picture to the Hermite-Einstein equation on a holomorphic vector bundle, with the novelty that we allow the solution $\mathbf{G}$ to have arbitrary signature. Since this is not completely standard in the literature, we give the details that apply in our setting. Let $X$ be a compact complex manifold endowed with a holomorphic vector bundle $\mathcal{Q}$. Recall that a choice of pseudo-Hermitian metric $\mathbf{G}$ on $\mathcal{Q}$ determines uniquely a compatible Chern connection $D^{\mathbf{G}}$ such that $(D^{\mathbf{G}})^{0,1}$ is the canonical Dolbeault operator on $\mathcal{Q}$.

\begin{lemma}\label{l:clasFut}
Let $X$ be a compact complex manifold, $\mathcal{Q}$ a holomorphic vector bundle over $X$, and $\mathfrak{b} \in H^{n-1,n-1}_{BC}(X,\RR)$ a Bott-Chern class. Then, the map
\begin{align}\label{eq:kahlerfutaki}
        \mathcal{F}_{\mathfrak{b}}:  H^{0}&(X,\mathrm{End}\; \mathcal{Q}) \longrightarrow \mathbb{C} \\
        & \varphi \mapsto \int_{X} \mathrm{tr}(\varphi F_{\mathbf{G}})\wedge \nu
\end{align}
defines a character of the Lie algebra $H^{0}(X,\mathrm{End}\; \mathcal{Q})$, which does not depend on the representative $\nu$ of $\mathfrak{b} = [\nu]$ and neither on the choice of a pseudo-Hermitian metric $\mathbf{G}$ on $\mathcal{Q}$. In particular, $\mathcal{F}_{\mathfrak{b}}=0$ if there exists a pseudo-Hermitian metric $\mathbf{G}$ on $\mathcal{Q}$ and a balanced Hermitian metric $\omega$ on $X$, with $\mathfrak{b} = [\omega^{n-1}]$, solving the Hermitian-Einstein equation
$$
F_{\mathbf{G}} \wedge \omega^{n-1} = 0.
$$
\end{lemma}
\begin{proof}
Let $\tilde{\nu}, \nu \in \Omega^{n-1,n-1}$ be $d$-closed forms on $X$, such that
$$
\tilde{\nu}-\nu = \overline{\partial}\partial \alpha
$$
for some $\alpha \in \Omega^{n-2,n-2}$. Then, by type decomposition and the Bianchi identity for $D^\mathbf{G}$, we have
\begin{align*}
\int_X \mathrm{tr}(\varphi F_{\mathbf{G}})\wedge \overline{\partial}\partial \alpha & = \int_{X}d(\mathrm{tr}(\varphi F_{\mathbf{G}})\wedge\partial \alpha) - \int_X \mathrm{tr}(\overline{\partial}\varphi F_{\mathbf{G}})\wedge \partial \alpha = 0,
\end{align*}
where the two summands vanish independently by hypothesis. Now, let $\mathbf{G}$ and $\mathbf{G}'$ be arbitrary pseudo-Hermitian metrics on $\mathcal{Q}$. Since $\mathbf{G}$ and $\mathbf{G}'$ are both non-degenerate, there exists a smooth complex gauge transformation $g$ on $\mathcal{Q}$ such that $\mathbf{G}'(\cdot,\cdot)= \mathbf{G}(g\cdot,\cdot)$. Then, their Chern curvatures are related by
$$
F_{\mathbf{G}'}=F_\mathbf{G}+\overline{\partial}(g^{-1}\partial^\mathbf{G}g)
$$
and it follows that, again by type decomposition and the holomorphicity of $\varphi$,
\begin{align*}
\int_{X} \mathrm{tr}(\varphi(F_{\mathbf{G}'}-F_\mathbf{G}))\wedge \nu & = \int_{X} \mathrm{tr}(\varphi\overline{\partial}(g^{-1}\partial^\mathbf{G}g))\wedge \nu\\
 & = \int_{X} d(\mathrm{tr}(\varphi (g^{-1}\partial^\mathbf{G}g))\wedge \nu)  -\int_X \mathrm{tr}(\overline{\partial}\varphi\wedge (g^{-1}\partial^\mathbf{G}g))\wedge \nu\\
 &  +\int_{X} \mathrm{tr}(\varphi (g^{-1}\partial^\mathbf{G}g))\wedge \overline{\partial}\nu\\
 & = 0.
\end{align*}
Finally, for $\varphi, \varphi' \in H^{0}(X,\mathrm{End}\; \mathcal{Q})$, using that $[F_\mathbf{G},\varphi'] = \dbar \partial^\mathbf{G}\varphi'$, one has
\begin{align*}
\int_{X} \mathrm{tr}([\varphi,\varphi'] F_\mathbf{G})\wedge \nu = %- \int_{X} \mathrm{tr}(\varphi [F_\mathbf{G},\varphi'])\wedge \nu = 
- \int_{X} \mathrm{tr}(\varphi\dbar \partial^\mathbf{G}\varphi')\wedge \nu = - \int_{X} d (\mathrm{tr}(\varphi \partial^\mathbf{G}\varphi'))\wedge \nu = 0.
\end{align*}
\end{proof}

Using the duality isomorphism $H^{n-1,n-1}_{BC}(X)^* \cong H^{1,1}_A(X)$, the \emph{Futaki invariants} in Lemma \ref{l:clasFut}, with $\mathfrak{b}$ varying along $H^{n-1,n-1}_{BC}(X)$, can be written more elegantly as a $H^{1,1}_A(X)$-valued character
$$
\mathcal{F} : H^{0}(X,\mathrm{End} \; \mathcal{Q}) \rightarrow H^{1,1}_A(X) \colon \varphi \mapsto [\tr(\varphi F_{\mathbf{G}})].
$$
In order to apply Lemma \ref{l:clasFut} to the Hull-Strominger system \eqref{eq:HSabstract}, we assume that the compact complex manifold $X$ is endowed with a holomorphic volume form $\Omega$. Let $V_0$ and $V_1$ be holomorphic vector bundles over $X$ satisfying \eqref{eq:c1c2abstract}. Denote by $P$ the holomorphic principal bundle of split frames of $V_0 \oplus V_1$. Fix $\alpha \in \mathbb{R}$ and consider the associated pairing $\IP{,}: \mathfrak{g} \otimes \mathfrak{g} \to \mathbb{C}$ on the Lie algebra of the structure group, defined as in \eqref{eq:pairingHS}. Then, we have
\begin{equation}\label{eq:BC22bis}
p_1(P) = 0 \in H^{2,2}_{BC}(X,\RR),
\end{equation}
and, by Proposition \ref{prop:BCclassification}, the set of equivalence classes of Bott-Chern algebroids over $X$ with principal bundle $P$ and bundle of quadratic Lie algebras $(\ad P,\IP{,})$ is a non-empty affine space $\mathfrak{S}^\alpha$ modelled on the image of
\begin{equation}\label{eq:partialmapbis}
\partial \colon H^{1,1}_A(X,\mathbb{R}) \to H^1(\Omega^{2,0}_{cl}).
\end{equation}
Consider the family of finite-dimensional complex Lie algebras
$$
\mathfrak{H}^\alpha \to \mathfrak{S}^\alpha,
$$
where the fibre over $\mathfrak{s} \in \mathfrak{S}^\alpha$ is given by the Lie algebra of the group of holomorphic gauge transformations of the vector bundle $\mathcal{Q}_{\mathfrak{s}}$
$$
\mathfrak{H}_\mathfrak{s}^\alpha := H^0(X,\End \mathcal{Q}_\mathfrak{s}).
$$
Then, by application of Lemma \ref{l:clasFut}, there is a family of $H^{1,1}_A(X)$-valued characters
\begin{equation}%\label{eq:FcharacterAe}
    \xymatrix{
 \mathcal{F}^\alpha \colon \mathfrak{H}^\alpha \ar[r] & H^{1,1}_A(X).
 }
\end{equation}

\begin{theorem}\label{t:Futaki}
Assume that $(X,\Omega,V_0,V_1)$ admits a solution $(g,h_0,h_1)$ of the Hull-Strominger system \eqref{eq:HSabstract} with coupling constant $\alpha \in \mathbb{R}$ and balanced class 
$$
\mathfrak{b} = [\|\Omega\|_\omega \omega^{n-1}] \in H^{n-1,n-1}_{BC}(X,\RR).
$$
Then, there exists $\mathfrak{s} \in\mathfrak{S}^\alpha$ such that $\IP{\mathcal{F}^\alpha_\mathfrak{s},\mathfrak{b}} = 0$.
\end{theorem}

\begin{proof}
Consider the Bott-Chern algebroid $Q_{P,2i\partial\omega,\theta^h}$ associated to the solution $(g,h_0,h_1)$, defined as in Example \ref{def:Q0}, where $h = h_0 \oplus h_1$. Denote by $\mathfrak{s} = [Q_{P,2i\partial\omega,\theta^h}] \in\mathfrak{S}^\alpha$ its isomorphism class. Then, by Proposition \ref{prop:HS-HE}, the (possibly) indefinite Hermitian metric $\mathbf{G}$ in Lemma \ref{t:Ggeneralized1} solves the Hermitian-Einstein equation \eqref{eq:HE}, and hence $\IP{\mathcal{F}^\alpha_\mathfrak{s},\mathfrak{b}} = 0$ by application of Lemma \ref{l:clasFut}.
\end{proof}

\begin{remark}
Following \cite{GaRuTi3}, we expect that the family of Lie algebras $\mathfrak{H}^\alpha$ depends holomorphically on parameters, upon restriction to any locus $\mathfrak{S}^\alpha_\sigma \subset \mathfrak{S}^\alpha$ with fixed \emph{real string class} $\sigma$.
\end{remark}

%[\textcolor{red}{Suggestion by M. Gualtieri: equip $X\times \mathfrak{S}$ with the weak topology of the projection $p: X\times \mathfrak{S}\rightarrow \mathfrak{S}$ (the open sets are vertical strips). Now regard $\{\mathcal{F}^\alpha\}_{\mathfrak{s}\in \mathfrak{S}}$ as a sheaf morphism $\mathcal{F}^\alpha: \mathcal{O}_{X\times \mathfrak{S}}(\mathrm{End}\; \mathbb{Q})\rightarrow p^{-1}\underline{H_A^{1,1}(X)}$ where $\mathbb{Q}$ is the universal Bott-Chern algebroid over $X\times \mathfrak{S}$ and $p^{-1}\underline{H_A^{1,1}(X)}$ is the projection map and $\underline{H^{1,1}_A(X)}$ is the  locally constant sheaf on $X$ with stalk $H^{1,1}_A(X)$. Moreover, as $\mathcal{O}_{X\times \mathfrak{S}}(\mathrm{End}\; \mathbb{Q})$ is reflexive, then $\mathcal{F}^\alpha=\mathcal{O}_{X\times \mathfrak{S}}(\mathrm{End}\; \mathbb{Q})\otimes_{\underline{\mathbb{C}}}p^{-1}\underline{H^{1,1}_A(X)}$ and it is a coherent sheaf.}]

As a direct application of Theorem \ref{t:Futaki}, we obtain Theorem \ref{t:Futakiintro} as stated in Section \ref{sec:intro}. In the case that the holomorphic tangent bundle $T^{1,0}$ (with the standard holomorphic structure) is polystable with respect to some balanced class $\mathfrak{b} \in H^{2,2}_{BC}(X,\mathbb{R})$, we expect that Theorem \ref{t:Futaki} provides also an obstruction to the existence of solutions to \eqref{eq:HSintro} with $\nabla = D^g$. %[\textbf{MGF: Debe haber una manera de construir una solucion con $\nabla$ HYM a partir de una solucion con nabla Chern, en estas condiciones.}]

As a consequence of Theorem \ref{t:Futaki}, in order to disprove the strong version of Yau's Conjecture in Question \ref{question} for the case of Calabi-Yau threefolds, it suffices to find a tuple $(X,\Omega,V)$, $\alpha \in \RR$ and a balanced class $\mathfrak{b} \in H^{2,2}_{BC}(X,\RR)$, such that $V$ is $\mathfrak{b}$-polystable and 
\begin{equation*}%\label{eq:obstruction}
\IP{\mathcal{F}_\mathfrak{s}^\alpha,\mathfrak{b}} \neq 0, \qquad \forall \mathfrak{s} \in \mathfrak{U}^0,
\end{equation*}
where $\mathfrak{U}^0$ denotes the restriction of the relative family of string algebroid extensions over a dense open subset of the moduli space for $V_0$.

When the Calabi-Yau manifold $X$ satisfies the $\partial\dbar$-Lemma the space $\mathfrak{S}^\alpha$ reduces to a point (see Proposition \ref{prop:BCclassification}). In this case we obtain a unique invariant $\mathcal{F}_0^\alpha$ obstructing the existence of solutions, which can be regarded as a \emph{stringy version} of the classical Futaki invariant for the holomorphic bundle $P$. Based on this, we expect that $\mathcal{F}_0^\alpha$ provides a useful tool to address Question \ref{question} in the case of Calabi-Yau manifolds obtained via conifold transitions and flops.

\begin{remark}\label{rem:BHEFutaki}
As a consequence of Proposition \ref{prop:BHE-HE} and Lemma \ref{l:clasFut}, we obtain a stronger version of Theorem \ref{t:Futaki}. For instance, let $P$ be a holomorphic principal bundle over a compact complex manifold $X$ which admits a solution $(g,h)$ of the the coupled Hermitian-Einstein system \eqref{eq:BHE}. Let $\tilde{g}$ be a Gauduchon metric in the conformal class of $g$. Then, if $\tilde{g}$ is balanced, then $\mathcal{F}_{\mathfrak{b}}=0$ where $\mathfrak{b}=[\tilde{\omega}^{n-1}]$.
\end{remark}

\subsection{Explicit formulae via anchored endomorphisms}\label{sec:Futakiformula}

Our next goal is to prove an explicit formula for the Futaki invariants in Theorem \ref{t:Futaki}. For this, we exploit further the structure of the Bott-Chern algebroid $\mathcal{Q}$ associated to a solution of the last equation in \eqref{eq:HSabstract}.   

We fix $P$ associated to a pair of holomorphic vector bundles $V_0$ and $V_1$ over $(X,\Omega)$, as in section \ref{sec:Futakidef}. Consider the Bott-Chern algebroid $\mathcal{Q} = \mathcal{Q}_{P,2i\partial\omega,\theta^h}$ (see Example \ref{def:Q0}) associated to a solution of
\begin{equation}\label{eq:BCalg}
    dd^c\omega+\IP{F_h \wedge F_h}  = 0,
\end{equation}
where $h = h_0 \oplus h_1$ and $\IP{,}$ is, as in \eqref{eq:pairingHS}, determined by a choice of $\alpha \in \mathbb{R}$. We will denote by $\Lambda^2 \mathcal{Q} \subset \mathrm{End} \ \mathcal{Q}$ the bundle of orthogonal endomorphisms of $\mathcal{Q}$, that is, satisfying
$$
\IP{\varphi \cdot , \cdot}_0 + \IP{\cdot , \varphi \cdot}_0 = 0.
$$

\begin{definition}
An element $\varphi \in \Gamma(\Lambda^2 \;\mathcal{Q})$ is called an \emph{anchored endomorphism} of $\mathcal{Q}$ if there exists $\phi \in \Gamma(\End T^{1,0})$ such that
\begin{equation}\label{eq:anchorcomp}
\pi \circ \varphi =\phi \circ \pi.
\end{equation}
\end{definition}

In our next result we provide an explicit characterization of anchored endomorphisms, via the identification of the smooth complex vector bundle underlying $\mathcal{Q}$ with $ T^{1,0}\oplus \mathrm{ad} P \oplus T^{*}_{1,0}$ (see Example \ref{def:Q0}).

\begin{lemma}\label{l:smoothEndQ}
Let $\mathcal{Q} = \mathcal{Q}_{P,2i\partial\omega,\theta^h}$ be the Bott-Chern algebroid associated to a solution of (\ref{eq:BCalg}). Let $\varphi$ be an anchored endomorphism of $\mathcal{Q}$. Then, there exists $\phi \in \Gamma(\End T^{1,0})$, $b\in \Omega^{2,0}$, $\sigma \in \Gamma(\Lambda^2 \mathrm{ad} P)$ skew-orthogonal, and $\alpha\in \Omega^{1,0}(\mathrm{ad} \hspace{1mm} P)$, uniquely determined by $\varphi$, such that
\begin{align}\label{eq:smoothEndQ}
\varphi= \varphi(\phi,\alpha,\sigma,b):= \left(\begin{array}{c c c}
\phi & 0 & 0\\
\alpha & \sigma & 0\\
b & -2\langle\alpha,\cdot\rangle & -\phi^*\\
\end{array}\right).
\end{align}
Conversely, any tuple $(\phi,b,\sigma,\alpha)$ as above defines an anchored endomorphism $\varphi$ of $\mathcal{Q}$ via formula \eqref{eq:smoothEndQ}. 
\end{lemma}
\begin{proof}
The proof follows directly from \cite[Section 3.1]{Ga0}.
\end{proof}

In our next result we characterize the holomorphicity condition $\dbar_0 \varphi = 0$, for $\varphi$ in \eqref{eq:smoothEndQ}, where $\dbar_0$ denotes the Dolbeault operator in Example \ref{def:Q0}.

\begin{lemma}\label{l:holEndQ}
Let $\mathcal{Q}\cong \mathcal{Q}_{P,2i\partial\omega,\theta^h}$ be the Bott-Chern algebroid associated to a solution of (\ref{eq:BCalg}). Let $\varphi = \varphi(\phi,\alpha,\sigma,b)$ be an anchored endomorphism of $\mathcal{Q}$. Then $\varphi$ is holomorphic if and only if the following conditions are satisfied
\begin{equation}\label{eq:holEndQ}
\begin{split}
    \overline{\partial}\phi & =0\\
    \overline{\partial}\sigma & = 0\\
    \overline{\partial}\alpha+\sigma(F_h)-F_h(\phi \cdot,\cdot) & =0\\
    \overline{\partial}b + \phi \lrcorner (2i\partial \omega) - 2\langle \alpha \wedge F_h\rangle & = 0
\end{split}
\end{equation}
where
\begin{align*}
i_{Y^{1,0}}i_{V^{1,0}}(\phi \lrcorner (2i\partial \omega)) & = i_{Y^{1,0}}i_{\phi(V^{1,0})}(2i\partial \omega) + i_{\phi(Y^{1,0})}i_{V^{1,0}}(2i\partial \omega)
\end{align*}
\end{lemma}
\begin{proof}
With the notation in Lemma \ref{l:smoothEndQ}, the proof follows from 
$$
(\overline{\partial}_{0}\varphi)(V+\xi+r) = \overline{\partial}_0(\varphi(V+\xi+r))-\varphi(\overline{\partial}_0(V+\xi+r))
$$
using the expression for $\overline{\partial}_0$ given by (\ref{eq:DolQ}). Imposing that this expression vanishes for any $V$, $\xi$ and $r$ is equivalent to the equations above.
\end{proof}

\begin{remark}
The subspace of solutions to \eqref{eq:holEndQ} with $\phi=\sigma=0$ is given by
$$
S_0 = \{(b,\alpha) \; | \; \overline{\partial}b-2\langle \alpha \wedge F_h\rangle = 0 \;, b \in \Omega^{2,0}, \; \alpha\in H^0(X,\Omega^{1,0}(\mathrm{ad} \hspace{1mm} P))\}.
$$
If we define (cf. \cite[Proposition 4.6]{GaRuTi2})
$$
\delta_P : H^0(X,\Omega^{1,0}(\mathrm{ad} \hspace{1mm} P)) \longrightarrow H^{2,1}_{\overline{\partial}}(X) \colon \alpha \mapsto  [2\langle \alpha \wedge F_h\rangle],
$$
the space $S_0$ fits in the short exact sequence
$$
0\rightarrow H^{2,0}_{\overline{\partial}}(X) \rightarrow S_0 \rightarrow \mathrm{ker} \hspace{1mm} \delta_P \rightarrow 0
$$
In particular if $h^{2,0}_{\overline{\partial}}(X) > 0$ or if $h^0(\Omega^{1,0}(\mathrm{ad} \hspace{1mm} P))>h^{2,1}_{\overline{\partial}}(X)$, then $\mathcal{Q}$ has a holomorphic anchored endomorphism.
\end{remark}

\begin{remark}
Suppose that $X$ satisfies the $\partial \dbar$-Lemma. Then, for any $s\in H^{0}(X,\mathrm{ad} \hspace{1mm} P)$,  we can construct holomorphic anchored endomorphisms $\varphi= \varphi(\phi,\alpha,\sigma,b)$, defined by \eqref{eq:smoothEndQ}, as follows: set
$$
\phi=0 \hspace{2mm}, \hspace{2mm} \alpha=\partial^hs \hspace{2mm}, \hspace{2mm} \sigma=[s,\cdot].
$$
Now, we have
$$
\dbar\langle \partial^hs\wedge F_h\rangle = \langle \dbar\partial^hs\wedge F_h\rangle = \langle [F_h,s]\wedge F_h\rangle = - \langle s\wedge [F_h \wedge F_h]\rangle = 0,
$$
and hence, by the $\partial \dbar$-Lemma, there exists $b \in \Omega^{2,0}$ such that 
$$
\overline{\partial}b=2\langle \partial^hs\wedge F_h\rangle,
$$
since the right hand side is $\partial$-exact and $d$-closed.
\end{remark}

Next, we address the computation of the Futaki invariants in Theorem \ref{t:Futaki} for holomorphic anchored endomorphism of $\mathcal{Q}$. For this, given a pair of Hermitian metrics $g$ and $g_0$ on $X$, $\gamma \in \Gamma(\End T \otimes \CC)$, and $\tau \in \Omega^2$, we denote
$$
\tr_{g,g_0} \gamma := \frac{1}{2}g(\gamma Je_j^0,e^0_j), \qquad \Lambda_{\omega_0} \tau = \frac{1}{2} \tau(e^0_j,Je^0_j), 
$$
for any choice of $g_0$-orthonormal basis $e_1^0, \ldots, e_{2n}^0$ of $T$, where we use Einstein's convention to sum over repeated indices.

\begin{proposition}\label{prop:Futakiexp}
Consider the Bott-Chern algebroid $\mathcal{Q} = \mathcal{Q}_{P,2i\partial\omega,\theta^h}$ associated to a solution $(\omega,h)$ of (\ref{eq:BCalg}), with $\omega$ positive, and pairing induced by \eqref{eq:pairingHS} for $\alpha \in \RR$. Let $\varphi = \varphi(\phi,\alpha,\sigma,b)$ be a holomorphic anchored endomorphism of $\mathcal{Q}$ and let $\mathfrak{b}=[\omega_0^{n-1}] \in H^{n-1,n-1}_{BC}(X,\RR)$ be a balanced class. Then, the evaluation of the Futaki character in Lemma \ref{l:clasFut} is given by
\begin{equation}\label{eq:Futakiexp}
\begin{split}
\langle \mathcal{F}^\alpha(\varphi), \mathfrak{b} \rangle & =  - \int_X (\tr_{g,g_0} R_{\nabla^B} + \IP{\Lambda_{\omega_0}F_h,F_h} )  (e_k,\phi^{*_g}e_k^{0,1} - \phi e_k^{1,0})\frac{\omega_0^n}{n}\\
& -\int_X \IP{\tr_{g,g_0} R_{\nabla^B}^{0,2},b}_g\frac{\omega_0^n}{n}\\
& + \int_X \left( \tr_{\ad P}(\sigma[\Lambda_{\omega_0}F_h,\cdot])+\langle \sigma F_h(e_j^0,e_k),F_h(Je_j^0,e_k)\rangle\right)\frac{\omega_0^n}{n}\\
& + 2 \int_X \IP{\alpha(e_k),\Lambda_{\omega_0}\nabla^{h,-}_{e_k}F_h + F_h(Je^0_j,g^{-1}d^c\omega(e_j^0,e_k,\cdot))}\frac{\omega_0^n}{n}\\
\end{split}
\end{equation}
for any choices of $g$-orthonormal basis $e_1, \ldots, e_{2n}$ and $g_0$-orthonormal basis $e_1^0, \ldots, e_{2n}^0$ of $T$.
\end{proposition}

\begin{proof}
Consider the isomorphism $\psi \colon \mathcal{Q}_{P,2i\partial\omega,\theta^h} \to T \otimes \CC \oplus\ad P$ defined by Lemma \ref{t:Ggeneralized1}, that is,
$$
\psi(V + r + \xi) = V - \tfrac{1}{2} g^{-1} \xi + r.
$$
Then $\tilde \varphi : = \psi \circ \varphi \circ \psi^{-1}$ is given by
$$
\tilde \varphi(V + r) =  \phi(V^{1,0}) - \phi^{*_g}(V^{0,1}) - \tfrac{1}{2} g^{-1}i_{V^{1,0}}b + g^{-1}\IP{\alpha,r} + \sigma(r) + i_{V^{1,0}}\alpha
$$
where $\phi^{*_g}(V^{0,1}) = g^{-1} g(V^{0,1},\phi\cdot)$. By Lemma \ref{l:clasFut} and formula \eqref{eq:secondRicciG}, it suffices to calculate
$$
\tr \tilde \varphi S_\mathbf{G} = \frac{1}{2} \tr \tilde \varphi F_\mathbf{G}(e_j^0,Je_j^0).
$$
For this, using that $\varphi = \varphi(\phi,\alpha,\sigma,b)$ depends linearly on $\phi$, $\alpha$, $\sigma$, and $b$, we can decompose uniquely 
$$
\tilde \varphi = \tilde \varphi_\phi + \tilde \varphi_b + \tilde \varphi_\alpha + \tilde \varphi_\sigma,
$$
so that $\tilde \varphi_\phi$ only depends on $\phi$, and similarly for the rest. Now, denoting $\pi_{1,0} \colon T \otimes \CC \to T^{1,0}$ the natural projection, by Lemma \ref{lemma:connectioncurvature} we have
\begin{align*}
\tr \tilde \varphi_\phi F_\mathbf{G}(e_j^0,Je_j^0) & = \tr_{T \otimes \CC} \phi \circ \pi_{1,0}(R_{\nabla^-}(e^0_j,Je^0_j)  - \mathbb{F}^\dagger \wedge \mathbb{F}(e^0_j,Je^0_j))\\
& - \tr_{T \otimes \CC} \phi^{*_g} \circ \pi_{0,1}(R_{\nabla^-}(e^0_j,Je^0_j)  - \mathbb{F}^\dagger \wedge \mathbb{F}(e^0_j,Je^0_j))\\
& = g(R_{\nabla^-}(e^0_j,Je^0_j)e_k,\phi^{*_g}e^{0,1}_k - \phi e^{1,0}_k) \\
& - \IP{F_h(Je^0_j,\phi^{*_g}e^{0,1}_k - \phi e^{1,0}_k),F_h(e_j^0,e_k)}\\
& + \IP{F_h(e^0_j,\phi^{*_g}e^{0,1}_k - \phi e^{1,0}_k),F_h(Je_j^0,e_k)}\\
& = \tfrac{1}{2}dd^c\omega (e^0_j,Je^0_j,e_k,\phi^{*_g}e^{0,1}_k - \phi e^{1,0}_k)\\
& + g(R_{\nabla^B}(e_k,\phi^{*_g}e^{0,1}_k - \phi e^{1,0}_k)e^0_j,Je^0_j)\\
& - \IP{F_h(Je^0_j,\phi^{*_g}e^{0,1}_k - \phi e^{1,0}_k),F_h(e_j^0,e_k)}\\
& + \IP{F_h(e^0_j,\phi^{*_g}e^{0,1}_k - \phi e^{1,0}_k),F_h(Je_j^0,e_k)}\\
& = - 2\tr_{g,g_0} R_{\nabla^B}(e_k,\phi^{*_g}e^{0,1}_k - \phi e^{1,0}_k)\\
& -2  \IP{\Lambda_{\omega_0}F_h,F_h(e_k,\phi^{*_g}e^{0,1}_k - \phi e^{1,0}_k)},
\end{align*}
where in the third and fourth equalities we have used \eqref{eq:BismutHulleq} and \eqref{eq:BCalg}, respectively. Similarly, 
\begin{align*}
\tr \tilde \varphi_b F_\mathbf{G}(e_j^0,Je_j^0) & = - \tfrac{1}{2}\tr_{T \otimes \CC} g^{-1}b \circ \pi_{1,0}(R_{\nabla^-}(e^0_j,Je^0_j)  - \mathbb{F}^\dagger \wedge \mathbb{F}(e^0_j,Je^0_j))\\
& = -\tfrac{1}{2}b(R_{\nabla^-}(e^0_j,Je^0_j)e_k,e_k)
+\tfrac{1}{2} b(g^{-1}\IP{F_h(Je^0_j,\cdot),F_h(e_j^0,e_k)},e_k)\\
& - \tfrac{1}{2}b(g^{-1} \IP{F_h(e^0_j,\cdot),F_h(Je_j^0,e_k)},e_k)\\
& = -\tfrac{1}{2}g(R_{\nabla^B}(e_k,e_m)e^0_j,Je^0_j)b(e_m,e_k)\\
& + \tfrac{1}{2} b(g^{-1}\IP{F_h(e^0_j,\cdot),F_h(e_j^0,e_k)},Je_k)\\ & - \tfrac{1}{2}b(g^{-1} \IP{F_h(e^0_j,\cdot),F_h(Je_j^0,e_k)},e_k)\\
& = - 2\IP{ \tr_{g,g_0} R_{\nabla^B}^{0,2},b}_g,
\end{align*}
and, using the notation in Lemma \ref{lemma:connectioncurvature},
\begin{align*}
\tr \tilde \varphi_\alpha F_\mathbf{G}(e_j^0,Je_j^0) & = 2\IP{\alpha(e_k),\mathbb{I}(e_j^0,Je_j^0)e_k}\\
%& = 2\IP{\alpha(e_k),\nabla^{h,-}_{e_k}F_h(e_j^0,Je_j^0)} + 4\IP{\alpha(e_k),F_h(Je^0_j,g^{-1}d^c\omega(e_j^0,e_k,\cdot))}\\
& = 4\IP{\alpha(e_k),\Lambda_{\omega_0}\nabla^{h,-}_{e_k}F_h} + 4\IP{\alpha(e_k),F_h(Je^0_j,g^{-1}d^c\omega(e_j^0,e_k,\cdot))}
\end{align*}
Finally, taking basis $\{r_k\}$ and $\{\tilde r_k\}$ of $\ad P$ such that $\IP{r_k,\tilde r_j} = \delta_{kj}$, we have 
\begin{align*}
\tr \tilde \varphi_\sigma F_\mathbf{G}(e_j^0,Je_j^0) & = \tr_{\ad P} \sigma \circ ([F_h(e_j^0,Je_j^0),\cdot] - \mathbb{F} \wedge \mathbb{F}^\dagger(e_j^0,Je_j^0))\\
& = \IP{\tilde r_k, \sigma([F_h(e_j^0,Je_j^0),r_k])}\\
& -  \IP{\tilde r_k, \sigma F_h(Je_j^0,g^{-1}\IP{F_h(e_j^0,\cdot),r_k})}\\
& +\IP{\tilde r_k \sigma F_h(e_j^0,g^{-1}\IP{F_h(Je_j^0,\cdot),r_k})}\\
& = 2\tr_{\ad P} \sigma([\Lambda_{\omega_0}F_h,\cdot]\\
& -  \IP{\tilde r_k, \sigma F_h(Je_j^0,e_k)}\IP{F_h(e_j^0,e_k),r_k}\\
& + \IP{\tilde r_k, \sigma F_h(e_j^0,e_k)}\IP{F_h(Je_j^0,e_k),r_k}\\
& = 2\tr_{\ad P} \sigma([\Lambda_{\omega_0}F_h,\cdot] + 2\IP{\sigma F_h(e_j^0,e_k),F_h(Je_j^0,e_k)}.
\end{align*}

\end{proof}

\subsection{The algebroid viewpoint}\label{sec:algebroid}

In this section we change our perspective and think of the Hull-Strominger system as defining a canonical geometry for a fixed Bott-Chern algebroid $\mathcal{Q}$ over a Calabi-Yau manifold $(X,\Omega)$. This was the approach taken in \cite{GaRuShTi}, where it was proved that these equations admit a variational interpretation by means of the \emph{dilaton functional}. To simplify the notation, we embrace for a moment the formalism of principal bundles. We fix a holomorphic principal $G$-bundle $P$ over $X$. In this setup, $h$ will denote a reduction of $P$ to a maximal compact subgroup $K \subset G$ and $\theta^h$ will denote the associated Chern connection. The basic idea is provided by the following definition:

\begin{definition}\label{def:HSonQ}
Let $\mathcal{Q}$ be a Bott-Chern algebroid over a Calabi-Yau manifold $(X,\Omega)$, with principal bundle $P$ and bundle of quadratic Lie algebras $(\ad P,\IP{,})$. Then, we say that a pair $(\omega,h)$, where $\omega$ is a Hermitian form on $X$ and $h$ as before, solving 
$$
dd^c\omega + \IP{F_h \wedge F_h} = 0,
$$
is a metric on $\mathcal{Q}$, if there is an isomorphism of Bott-Chern algebroids
$$
\mathcal{Q}_{P,2i\partial\omega,\theta^h} \cong \mathcal{Q},
$$
where $\mathcal{Q}_{P,2i\partial\omega,\theta^h}$ is defined as in Example \ref{def:Q0}. Furthermore, we say that $(\omega,h)$ solves the Hull-Strominger system on $\mathcal{Q}$ if
\begin{equation}\label{eq:HSQ}
\begin{split}
F_h\wedge \omega^{n-1} & = 0,\\
d(\|\Omega\|_\omega \omega^{n-1} ) & = 0.
\end{split}
\end{equation}
\end{definition}

Notice that the existence of Hermitian metrics on a given $\mathcal{Q}$ is, a  priori, a difficult question (cf. Example \ref{ex:h-19} and \cite[Proposition 3.15]{GaRuShTi}). Without loss of generality, we can assume that $\mathcal{Q} = \mathcal{Q}_{P,2i\partial\tau_0,\theta^{h_0}}$ for a pair $(\tau_0,h_0)$ solving $dd^c\tau_0 + \IP{F_{h_0} \wedge F_{h_0}} = 0$. Here, $\tau_0$ is a real $(1,1)$-form on $X$ which may not be positive definite. Then, the condition for $(\omega,h)$ to be a metric on $\mathcal{Q}$ can be written more explicitly as
$$
db  = 2i\partial \omega - 2i\partial\tau_0 + 2 \langle a \wedge F_{h_0} \rangle + \langle a \wedge d^{\theta^{h_0}}a\rangle + \frac{1}{3} \IP{a \wedge [a \wedge a]},
$$
for $a = \theta^h - \theta^{h_0} \in \Omega^{1,0}(\ad P)$ and some $b \in \Omega^{2,0}$.

Our next result gives an obstruction to the existence of solutions of the Hull-Strominger system on a given Bott-Chern algebroid.

\begin{theorem}\label{thm:FutakionQ}
Let $\mathcal{Q}$ be a Bott-Chern algebroid over a Calabi-Yau manifold $(X,\Omega)$. Consider the associated Futaki invariant
$$
\mathcal{F} : H^{0}(X,\mathrm{End} \; \mathcal{Q}) \rightarrow H^{1,1}_A(X)
$$
constructed in Lemma \ref{l:clasFut}. Then, if there exists a solution of the Hull-Strominger system \eqref{eq:HSQ} on $\mathcal{Q}$ with balanced class $\mathfrak{b} \in H^{n-1,n-1}_{BC}(X,\RR)$, it follows that
$$
\langle \mathcal{F}, \mathfrak{b} \rangle = 0.
$$
\end{theorem}
\begin{proof}
The proof follows, as in the proof of Theorem \ref{t:Futaki}, by Proposition \ref{prop:HS-HE} and Lemma \ref{l:clasFut}.
\end{proof}

Recall that a given Bott-Chern algebroid has an associated affine space of \emph{real Aeppli classes} $\Sigma_\mathcal{Q}(\RR)$ (see \cite[Definition 3.20]{GaRuShTi}), modelled on the Kernel of \eqref{eq:partialmap}. The conjectural picture which was put forward by the first author jointly with Rubio, Shahbazi, and Tipler, is that one should expect uniqueness of solutions of the Hull-Strominger on each Aeppli class up to automorphisms of $\mathcal{Q}$ (cf. \cite[Appendix A]{GaRuShTi}). If true, this yields a well-defined map
$$
\Sigma_\mathcal{Q}(\RR) \supset U \longrightarrow \{\mathfrak{b} \; | \; \langle \mathcal{F}, \mathfrak{b} \rangle = 0 \} \subset H^{n-1,n-1}_{BC}(X,\RR)
$$
on the locus $U \subset \Sigma_\mathcal{Q}(\RR) $ of Aeppli classes which admit a solution. It is an interesting open question to see if the previous (conjectural) map can be extended to $\Sigma_\mathcal{Q}(\RR)$. This would yield a more transparent obstruction associated to the Futaki invariant, defined in terms of the given Bott-Chern algebroid and a choice of Aeppli class.

\end{document}